\documentclass[12pt]{amsart}
\usepackage{amssymb, a4wide, mathdots,url,hyperref}
\usepackage[usenames]{color}

\newcommand{\grp}[1]{{\mathbf{#1}}}

\newcommand{\C}{\mathbb{C}}
\newcommand{\Z}{\mathbb{Z}}
\newcommand{\Q}{\mathbb{Q}}
\newcommand{\R}{\mathbb{R}}

\renewcommand{\Re}{\operatorname{Re}}

\newcommand{\mult}{\mathbb{G}_m}
\newcommand{\add}{\mathbb{G}_a}
\newcommand{\inv}{\theta}
\newcommand{\roots}{\Sigma}
\newcommand{\Ad}{\operatorname{Ad}}
\newcommand{\dnull}{\Delta^G[\inv=-1]}
\newcommand{\dnotnull}{\Delta^G[\inv\ne-1]}
\newcommand{\proj}{{\mathbf{p}}}
\newcommand{\set}{\mathcal{S}}
\newcommand{\cone}{\mathcal{C}}
\newcommand{\ccone}{\overline{\mathcal{C}}}
\newcommand{\Exp}{\mathcal{E}}
\newcommand{\E}{\mathcal{E}}
\newcommand{\mc}{c}
\newcommand{\MC}{\mathcal{M}}

\newcommand{\Lie}{\operatorname{Lie}}
\newcommand{\Tr}{\operatorname{Tr}}

\newcommand{\Hom}{\operatorname{Hom}}
\renewcommand{\subset}{\subseteq}

\newcommand{\bs}{\backslash}

\newcommand{\diag}{\operatorname{diag}}

\newcommand{\vol}{\operatorname{vol}}
\newcommand{\GL}{\operatorname{GL}}
\renewcommand{\O}{\operatorname{O}}
\newcommand{\U}{\operatorname{U}}
\newcommand{\Sp}{\operatorname{Sp}}

\newcommand{\abs}[1]{\left|{#1}\right|}
\newcommand{\aaa}{\mathfrak{a}}

\newcommand{\Res}{\operatorname{Res}}

\newcommand{\h}{\mathcal{H}}

\newcommand{\sprod}[2]{\left\langle#1,#2\right\rangle}

\newcommand{\gl}{\mathfrak{gl}}
\newcommand{\pos}{>0}
\newcommand{\wpos}{\ge0}

\newcommand{\sm}[4]{{\bigl(\begin{smallmatrix}{#1}&{#2}\\{#3}&{#4}
\end{smallmatrix}\bigr)}}

\newtheorem{theorem}{Theorem}[section]
\newtheorem{lemma}[theorem]{Lemma}
\newtheorem{proposition}[theorem]{Proposition}
\newtheorem{remark}[theorem]{Remark}

\newtheorem{definition}[theorem]{Definition}
\newtheorem{corollary}[theorem]{Corollary}

\title[A criterion for integrability of matrix coefficients]{A criterion for integrability of matrix coefficients with respect to a symmetric space}
\author{Maxim Gurevich}
\address{Department of Mathematics, Technion -- Israel Institute of Technology , Haifa 3200003, Israel}
\email{max@tx.technion.ac.il}
\author{Omer Offen}
\address{Department of Mathematics, Technion -- Israel Institute of Technology , Haifa 3200003, Israel}
\email{offen@tx.technion.ac.il}
\date{\today}

\begin{document}

\setcounter{tocdepth}{1}

\begin{abstract}
Let $G$ be a reductive group and $\theta$ an involution on $G$, both defined over a $p$-adic field.
We provide a criterion for $G^\theta$-integrability of matrix coefficients of representations of $G$ in terms of their exponents along $\theta$-stable parabolic subgroups. 
The group case reduces to Casselman's square-integrability criterion. As a consequence we assert that certain families of symmetric spaces are strongly tempered in the sense of Sakellaridis and Venkatesh. For some other families our result implies that matrix coefficients of all irreducible, discrete series representations are $G^\theta$-integrable. 
\end{abstract}

\maketitle
\tableofcontents
\section{Introduction}

Let $F$ be a $p$-adic field. Let $G$ be the group of $F$-points of a reductive $F$-group, $\inv$ an involution on $G$ and $H=G^\inv$ the subgroup of $\inv$-fixed points.
In this work we provide a criterion for $H$-integrability of matrix coefficients of admissible representations of $G$ in terms of their exponents along $\inv$-stable parabolic subgroups of $G$. In the group case ($G=H\times H$, $\inv(x,y)=(y,x)$) our result reduces to Casselman's square-integrability criterion \cite[Theorem 4.4.6]{CassNotes}. 

For a smooth representation $\pi$ of $G$, let $\Hom_H(\pi,\C)$ be the space of $H$-invariant linear forms on $\pi$. 
As apparent, for example, from the general treatment of \cite{MR1075727}, this space plays an essential role in the harmonic analysis of the space $G/H$. See also \cite{MR2401221} for the study of $H$-invariant linear forms on induced representations in the context of a $p$-adic symmetric space and \cite{1203.0039} in the more general setting of a spherical variety.

Furthermore, the understanding of $H$-invariant linear forms in the local setting has applications to the study of period integrals of automorphic forms. A conjecture of Ichino-Ikeda \cite{MR2585578} treats a different setting in which the pair $(G,H)$ is of the Gross-Prasad type. It claims, roughly speaking, that under appropriate assumptions, the Hermitian form on an irreducible, tempered, automorphic representation of $G$ associated to the absolute value squared of the $H$-period integral factorizes as a product of local $H$-integrals of the associated matrix coefficients. The conjectural framework of \cite{1203.0039} suggests a generalization of this phenomenon, which will include the symmetric case. (For an explicit factorization of a somewhat different nature see e.g. \cite{MR1816039, MR2930996}.)

Integrability of matrix coefficients provides an explicit construction of the local components of period integrals of automorphic forms. Factorizable period integrals, in turn, are intimately related with special values of $L$-functions and with Langlands functoriality conjectures.

The above global conjectures suggest to study the following purely local questions. Let $A_G$ be the maximal split torus in the centre of $G$ and $A_G^+$ the connected component of its intersection with $H$. Let $\pi$ be a smooth representation of $G$ and $\tilde v$ a smooth linear form in its contragredient $\tilde\pi$.
\begin{itemize}
\item Is the linear form
\[
\ell_{\tilde v,H}(v):=\int_{H/A_G^+} \tilde{v}(\pi(h)v)\ dh
\]
well defined on $\pi$ by an absolutely convergent integral?  
(When this is the case $\ell_{\tilde v,H}\in \Hom_H(\pi,\C)$.)
\item
Is it non-zero?
\end{itemize}
The answer we provide for the first question is a relative analogue of Casselman's criterion. We recall that, essentially, that criterion says that an admissible representation $\pi$ of $G$ is square-integrable if and only if all its exponents are positive.
The two main ingredients in its proof are:
\begin{enumerate}
\item The Cartan decomposition of $G$, which allows to test convergence of a $G$-integral by convergence of a series summed over a positive cone in the lattice associated with a maximal split torus in $G$.
\item Casselman's pairing, which is a tool to study the asymptotics of matrix coefficients in a positive enough cone in terms of its Jacquet modules along parabolic subgroups and eventually, in terms of the exponents of the representation.
\end{enumerate}
Similarly, testing $H$-integrability, can be put in terms of convergence of a series summed over a positive cone in a maximal split torus in $H$.
In order to apply the asymptotics of matrix coefficients of representations of $G$ one has to relate positivity of the cone in $H$ to positivity of relevant cones in $G$.
We achieve this by further studying a root system, introduced by Helminck-Wang, associated to a symmetric space $G/H$ \cite[Proposition 5.7]{MR1215304}. It is a root system containing that of $H$ that we refer to as the \emph{descendent root system}. A key ingredient in our proof is the relation, obtained in Corollary \ref{cor: cone}, between the two notions of positivity.

In what follows we explicate our main result. 
Let $P_1$ be a minimal $\inv$-stable parabolic subgroup of $G$ and $P_0$ a minimal parabolic subgroup of $G$ contained in $P_1$. There exists a maximal split torus $A_0$ of $G$ in $P_0$ that is $\inv$-stable. Let $\aaa_0^*=X^*(A_0)\otimes_\Z \R$ where $X^*(A_0)$ is the lattice of $F$-characters of $A_0$. Then $\inv$ acts as an involution on $\aaa_0^*$ and gives rise to a decomposition 
\[
\aaa_0^*=(\aaa_0^*)^+_\inv\oplus (\aaa_0^*)^-_\inv
\]
where $(\aaa_0^*)^\pm_\inv$ is the $\pm 1$-eigenspace of $\inv$.

Let $P$ be a parabolic subgroup of $G$ containing $P_0$ (a standard parabolic subgroup) with a standard Levi decomposition $P=M\ltimes U$ and let $A_M$ be the maximal split torus in the centre of $M$. Then $\aaa_0$ admits a decomposition
\[
\aaa_0^*=\aaa_M^* \oplus (\aaa_0^M)^*
\]
where $\aaa_M^*=X^*(A_M)\otimes_\Z \R$ (see section \ref{notation} for details).
Assume that $P$ (and therefore also $M$) is $\inv$-stable. Then $\inv$-acts on $\aaa_M^*$ as an involution and decomposes it into the $\pm1$-eigenspaces 
\[
\aaa_M^*=(\aaa_M^*)_\inv^+\oplus(\aaa_M^*)_\inv^-.
\]
Let 
\[
\lambda\mapsto (\lambda_M)_\inv^+:\aaa_0^* \rightarrow (\aaa_M)_\inv^+ 
\]
be the projection to the first component with respect to the decomposition
\[
\aaa_0^*=(\aaa_M^*)_\inv^+\oplus(\aaa_M^*)_\inv^- \oplus (\aaa_0^M)^*.
\]
Let $A_M^+$ be the connected component of $A_M^\inv$. Then $(\aaa_M^*)_\inv^+\simeq X^*(A_M^+)\otimes_\Z \R$ and in particular $(\aaa_0^*)_\inv^+\simeq X^*(A_0^+)\otimes_\Z \R$. 

Let $\roots^G$ be the root system  of $G$ with respect to $A_0$ and let $\Delta$ be the set of simple roots determined by $P_0$.
Let $\Delta^{G/H}(M)$ be the set of non-zero restrictions to $A_M^+$ of the elements of $\Delta$. We say that $\lambda\in \aaa_0^*$ is $M$-relatively positive if $(\lambda_M)_\inv^+$ is a linear combination of the elements of $\Delta^{G/H}(M)$ with positive coefficients. 

There are two other root systems relevant to our main result. The root system $\roots^H$ of $H$ with respect to $A_0^+$ and the descendent root system $\roots^{G/H}$ which is the set of roots of $A_0^+$ in $\Lie(G)$. Let $W^H$ and $W^{G/H}$ be the associated Weyl groups. By definition, $\roots^H\subseteq \roots^G$ and this induces the imbedding $W^H\subseteq W^{G/H}$. In Corollary \ref{cor: cone}\eqref{part: cone deco} we define a particular set of representatives $[W^{G/H}/W^H]$ for the coset space $W^{G/H}/W^H$. 

Let $\rho_0^G\in \aaa_0^*$ be the usual half sum of positive roots in $\roots^G$ (summed with multiplicities). Note that similarly, $\rho_0^H\in (\aaa_0^*)_\inv^+$ and that $W^{G/H}$ acts on $(\aaa_0^*)_\inv^+$.
Our main result takes the following form.
\begin{theorem}\label{main intro}
Let $\pi$ be an admissible representation of $G$. Then every matrix coefficient of $\pi$ is $H$-integrable if and only if for every $\inv$-stable, standard parabolic subgroup $P=M\ltimes U$ of $G$, any exponent $\chi$ of $\pi$ along $P$ and any $w\in [W^{G/H} / W^H]$ we have that $\rho_0^G-2w\rho_0^H+\Re(\chi)$ is $M$-relatively positive. 
\end{theorem}
For the definition of exponents of admissible representations see Section \ref{exp}. For the definition of $\Re(\chi)\in \aaa_M^*$ for a character $\chi$ of $A_M$ see \eqref{eq: def Re}.

Following Sakellaridis-Venkatesh \cite{1203.0039}, we say that $G/H$ is strongly tempered (resp.~strongly discrete) if matrix coefficients of irreducible, tempered (resp.~discrete series) representations of $G$ are all $H$-integrable.

Pairs of the Gross-Prasad type are strongly tempered by \cite{MR2585578} in the special orthogonal case and \cite{MR3159075} in the unitary case. As a consequence of the general criterion obtained in this work, we provide in section \ref{sect: special cases} examples of symmetric spaces that are strongly tempered or at least strongly discrete. We recapitulate the results here.

\begin{corollary}[of Theorem \ref{thm: main}]
Let $E/F$ be a quadratic extension and $J\in \GL_n(F)$ a symmetric matrix. 

In the following cases $G/H$ is strongly tempered:
\[
\begin{array}{| c | c |}
\hline G & H \\ \hline
\GL_n(F) & \O_J(F) \\ \hline

\U_{J, E/F}(F) &  \O_J(F) \\ \hline

\Sp_{2n}(F) & \U_{J, E/F}(F) \\ \hline
\GL_2(F) & \GL_1(F)\times \GL_1(F) \\ \hline

\end{array}
\]
Here $\O_J$ is the orthogonal group associated to $J$ and $\U_{J,E/F}$ the unitary group associated to $J$ and $E/F$.

In the following cases $G/H$ is strongly discrete:
\[
\begin{array}{| c | c |}
\hline G & H \\ \hline

G'(E) &  G'(F) \\ \hline
\GL_{2n}(F) & GL_n(E) \\ \hline
\GL_{2n}(F) & \GL_n(F)\times \GL_n(F) \\ \hline
\GL_{2n+1}(F) & \GL_n(F)\times \GL_{n+1}(F) \\ \hline

\end{array}
\]
Here $G'$ is any reductive group defined over $F$.

\end{corollary}

For real symmetric spaces it is shown in \cite{1211.1203} that weak positivity of $\rho_0^G-2\rho_0^H$ is equivalent to $L^2(G/H)$ being tempered.
It will be interesting to study the relation between temperedness of $L^2(G/H)$ and the above properties, strongly tempered/discrete, in the $p$-adic case.

When $G$ is split over $F$, Sakellaridis and Venkatesh show in  \cite{1203.0039} that if $G/H$ is strongly tempered then all $H$-invariant linear forms of an irreducible, square-integrable representation $\pi$ of $G$ emerge as $H$-integrals of matrix coefficients, i.e., 
\[
\Hom_H(\pi,\C)=\{\ell_{\tilde v,H}:\tilde{v}\in \tilde \pi\}. 
\]
We apply this result in section \ref{sect: non-vanish} to some examples of symmetric spaces that are strongly tempered by our criterion. This expands on some similar recently obtained results. Pairs of Gross-Prasad type are strongly tempered and of multiplicity one. For those cases, the non-vanishing of $H$-integrals of matrix coefficients was obtained in \cite[Proposition 5.6]{MR3202558} and \cite[Theorem 14.3.1]{1205.2987}.
For irreducible cuspidal representations it is shown in \cite{1410.8274} for all symmetric spaces that all $H$-invariant linear forms emerge as $H$-integrals of matrix coefficients.
This is a generalization of \cite[\S 5]{MR1816039}.

The paper is organized as follows. After setting our notation in section \ref{notation}, we recall in section \ref{prelim} some basic facts about symmetric spaces.
In particular we recall the definition of the descendent root system associated to a symmetric space $G/H$ by Helminck and Wang and prove some relations with the root systems of $G$ and of $H$ that are relevant to the rest of this work. In section \ref{Hint} we prove our main result, a criterion for $H$-integrability of matrix coefficients. 
In section \ref{sect: special cases} we provide examples of strongly tempered/discrete symmetric spaces based on our main result. In section \ref{sect: non-vanish} we apply results of Sakellaridis and Venkatesh to provide examples where $H$-invariant linear forms emerge as integrals of matrix coefficients.

\subsection*{Acknowledgements}
We thank Yiannis Sakellaridis for sharing his insights on strongly tempered spaces.



\section{Notation}\label{notation}
Let $F$ be a $p$-adic field.
In general, if $\grp X$ is an algebraic variety defined over $F$ (an $F$-variety) we write $X=\grp X(F)$ for its $F$-points.

Let $\grp G$ be an algebraic $F$-group and $\grp{A_G}$ the maximal $F$-split torus in the centre of $\grp G$. We denote by $X^*(G)$ the group of $F$-rational characters of $\grp G$.
Let $\aaa_G^*=X^*(G)\otimes_{\Z}\R$ and let $\aaa_G=\Hom_{\R}(\aaa_G^*,\R)$ be its dual vector space with the natural pairing
$\sprod{\cdot}{\cdot}=\sprod{\cdot}{\cdot}_G$. We have $\aaa_G^*=\aaa_{A_G}^*$.

To $\lambda\otimes a\in \aaa_G^*$ we associate the character $g\mapsto \abs{\lambda(g)}^a$ of $G$. This extends to a bijection from $\aaa_G^*$
to the group of positive continuous characters of $G$. We denote by $\Re(\chi)\in \aaa_G^*$ the pre-image of a positive character $\chi:G\rightarrow \R_{>0}$.
If $\chi:G\rightarrow \C^*$ is any continuous homomorphism then we set 
\begin{equation}\label{eq: def Re}
\Re(\chi)=\Re(\abs{\chi}).
\end{equation}

Let $X_*(G)$ be the set of one parameter subgroups of $G$ (i.e., $F$-homomorphisms $\mult\rightarrow \grp G$). For an $F$-torus $\grp T$, $X_*(T)$ is a free abelian group of finite rank. The natural pairing of $X_*(T)$ with $X^*(T)$ allows us to identify $\aaa_T$ with $X_*(T)\otimes_\Z\R$.


Let $\delta_G$ be the modulus function of $G$\footnote{ Our convention will be that if $dg$ is a left-invariant Haar measure, then $\delta_G(g) dg$ is a right-invariant Haar measure. This is opposite to $\Delta_G$ in the convention of Bourbaki.}. 

From now on assume that $\grp G$ is a connected reductive group. Let $\grp A_G$ be the maximal $F$-split torus in the centre of $\grp G$. 

Let $\grp P_0=\grp M_0\ltimes \grp U_0$ be a minimal parabolic $F$-subgroup of $\grp G$ with Levi component $\grp M_0$ and unipotent radical $\grp U_0$.
Set $\grp A_0=A_{M_0}$, $\aaa_0=\aaa_{M_0}$ and $\aaa_0^*=\aaa_{M_0}^*$. Then $\grp A_0$ is a maximal $F$-split torus in $\grp G$. 

 
A parabolic $F$-subgroup $\grp P$ of $\grp G$ is called semi-standard if it contains $\grp A_0$, and standard if it contains $\grp P_0$. If $\grp P$ is semi-standard, it admits a unique Levi subgroup $\grp M$ containing $\grp A_0$. We will say that $\grp M$ is a semi-standard Levi subgroup of $\grp G$. When we write that $\grp P =\grp M\ltimes \grp U$ is a semi-standard parabolic $F$-subgroup of $\grp G$, we will mean that $\grp M$ is the unique semi-standard Levi subgroup of $\grp P$ and $\grp U$ is the unipotent radical of $\grp P$.

The space $\aaa_M$ can be identified with $\aaa_{A_M}$, and in particular can be viewed as a subspace of $\aaa_0= \aaa_{A_0}$ with a canonical decomposition
 \[
 \aaa_0=\aaa_M \oplus \aaa_0^M.
 \] 
More generally, if $\grp Q=\grp L \ltimes \grp V$ is another semi-standard $F$-parabolic subgroup of $\grp G$ containing $\grp P$ then $\aaa_L$ is a subspace of $\aaa_M$ and there is a canonical decomposition
\[
\aaa_M=\aaa_L \oplus \aaa_M^L.
\]
Similar decompositions apply to the dual spaces. 
For $\lambda\in \aaa_0^*$ we denote by $\lambda_M$ its projection to $\aaa_M^*$ and by $\lambda_M^L$ its projection to $(\aaa_M^L)^*$.

Let $\grp T$ be an $F$-split torus in $\grp G$.
If $0\ne v\in \Lie(G)$ and $0\ne \alpha\in X^*(T)$ are such that $\Ad(t)v=\alpha(t)v$, $t\in \grp T$ then we say that $\alpha$ is a root of $G$ with respect to $T$ and $v$ is a root vector with root $\alpha$. Let $R(T,G)$ be the set of all roots of $G$ with respect to $T$.

Let $\roots=\roots^G=R(A_0,G)$. It is a subset of $X^*(A_0)$ that spans $(\aaa_0^G)^*$ and forms  a root system. Let $\roots^{\pos}=\roots^{G,\pos}=R(A_0,P_0)$ be the set of positive roots and $\Delta=\Delta^G$ the basis of simple roots with respect to $P_0$. 
Let $W^G$ denote the Weyl group of $\roots^G$.
For a standard parabolic subgroup $\grp P=\grp M\ltimes \grp U$ of $\grp G$ let $\Delta^M=\Delta\cap \roots^M$ be the set of simple roots of $M$ with respect to $M\cap P_0$.
Furthermore, let
\[
\Delta_M=\{\alpha|_{A_M}:\alpha\in \Delta^G\}\setminus \{0\}.
\]

For $\lambda\in X_*(G)$ we associate a parabolic $F$-subgroup $\grp P(\lambda)=\grp{P_G}(\lambda)$ as in \cite[\S 15.1]{MR2458469}.
It is defined as the set of points $x\in \grp G$ so that the map $a\mapsto \lambda(a)x\lambda(a)^{-1}:\mult\rightarrow \grp G$ extends to an $F$-rational map 
$\add\rightarrow \grp G$. (Here we view the multiplicative group $\mult$ as a subvariety of the additive group $\add$.) It naturally comes with a Levi decomposition $\grp P(\lambda)=\grp M(\lambda) \ltimes \grp U(\lambda)$ where the Levi component $\grp M(\lambda)$ is the centralizer of the image of $\lambda$ and the unipotent radical consists of the elements $x$ where the above extended map sends $0$ to the identity in $\grp G$. The group $\grp P(-\lambda)$ is the parabolic subgroup of $\grp G$ opposite to $\grp P(\lambda)$ so that $\grp P(\lambda)\cap \grp P(-\lambda)=\grp M(\lambda)$.
Every parabolic $F$-subgroup of $\grp G$ is of the form $\grp P(\lambda)$ for some $\lambda\in X_*(G)$ (see \cite[Lemma 15.1.2]{MR2458469}).

Furthermore, every semi-standard parabolic $F$-subgroup of $\grp G$ is of the form $\grp P(\lambda)$ where $\lambda\in X_*(A_0)$.
(In fact, semi-standard parabolic $F$-subgroups of $\grp G$ are in bijection with facets of $\aaa_0$ with respect to root hyperplanes associated to $\roots$.)

For a subset $I\subseteq \Delta$ let $\lambda_I\in X_*(A_0)$ be such that $\sprod{\alpha}{\lambda_I}=0$ for all $\alpha\in I$ and $\sprod{\alpha}{\lambda_I}>0$ for all $\alpha\in \Delta\setminus I$. Then $\grp {P_I}:=\grp P(\lambda_I)$ is a standard parabolic $F$-subgroup of $\grp G$. In fact, $\grp{P_I}$ is independent of a choice of $\lambda_I$ as above and $I\mapsto \grp{P_I}$ is an order preserving bijection between subsets of $\Delta$ and standard parabolic $F$-subgroups of $\grp G$. We denote by $\grp{P_I}=\grp{M_I}\ltimes\grp{U_I}$ the associated Levi decomposition and let $\grp{A_I}=\grp{A_{M_I}}$. Then $A_I$ is the connected component of $\cap_{\alpha\in I} \ker \alpha \subset A_0$ and $\Delta^{M_I}=I$.
Note that $P_\emptyset=P_0$ and $P_{\Delta}=G$.

\subsection{Cones} Let $\grp T$ be an $F$-split torus. For a subset $\set\subseteq X^*(T)$ let 
\[
\aaa_T^{\set,\pos}=\{x\in \aaa_T: \sprod{\alpha}{x}>0,\ \alpha\in \set\},
\] 
$\aaa_T^{\set,\wpos}$ be its closure and 
\[
X_*(T)^{\set,\wpos}=X_*(T)\cap \aaa_T^{\set,\wpos}. 
\]
Also let 
\[
\cone(T,\set)=\{\sum_{\alpha\in \set} a_\alpha \alpha: a_\alpha\in \R_{>0},\ \alpha\in \set\}
\]
and let $\ccone(T,\set)$ be its closure.

Fix a uniformizer $\varpi$ of $F$ once and for all. Then, $X_*(T)$ can be embedded in $T$ by $x\mapsto x(\varpi)$. We denote the image of this embedding by $C_T$. Then $T/C_T$ is compact. Let 
$C_T^{\set,\wpos}$ be the image of $X_*(T)^{\set,\wpos}$ in $C_T$. 

Let $\grp P=\grp M \ltimes\grp U$ be a standard parabolic $F$-subgroup of $\grp G$.
For $\epsilon>0$ let 
\[
C_{A_M}^{\pos}(\epsilon)=\{a\in C_{A_M}: \abs{\alpha(a)}_F<\epsilon,\, \alpha\in \Delta_M\}.
\]
Note that if $\epsilon \le 1$ then $C_{A_M}^{\pos}(\epsilon)\subseteq C_{A_M}^{\Delta_M,\wpos}$.

\subsection{Cartan decomposition}
Let 
\[
C_0^{\wpos}=C_{A_0}^{\Delta^G,\wpos} 
\]
and fix a maximal compact subgroup $K=K^G$ of $G$ `adapt\`e \'a $A_0$' in the terminology of \cite[\S V.5.1]{MR2567785}.
 By our choice of $K$ (see \cite[Theorem V.5.1(4)]{MR2567785}) there exists a finite set $F_0$ in $M_0$ such that 
\[
G=\bigsqcup_{c\in C_0^{\wpos}} \bigsqcup_{f\in F_0} KfcK.
\]
Fix a Haar measure on $G$ and let $\vol(X)$ denote the measure of a subset $X$ of $G$. Choosing the set $F_0$ as in \cite[Theorem V.3.21]{MR2567785} the following follows from  \cite[Theorem V.5.2]{MR2567785} and the proof of  \cite[Theorem VII.1.2]{MR2567785}\footnote{We recall that our convention of modulus function is opposite to that of Renard.}.
\begin{lemma}\label{lem: vol cart}
There exists a basis $\mathcal{I}$ of neighbourhoods of the identity in $G$ consisting of open normal subgroups of $K$ such that
\[
\vol(K_0 fc K_0)=\delta_{P_0}^{-1}(fc) \vol(K_0)
\]
for all $K_0\in\mathcal{I}$.
\end{lemma}

\subsection{The symmetric subgroup}

Let $\inv$ be an involution on $\grp G$ defined over $F$ and 
\[
\grp H=\grp G^\inv=\{g\in \grp G: \theta(g)=g\}.
\] 
We further denote by $\inv$ the differential of its action on $G$. It is an involution on $\Lie(G)$ and 
\begin{equation}\label{eq: lie inv}
\Lie(H)=\Lie(G)^\inv.
\end{equation}

Let $\grp H^\circ$ be the connected component of the identity in $\grp H$. It is a connected reductive $F$-group and $\grp H^\circ$ is of finite index in $\grp H$. 

For a $\inv$-stable $F$-torus $\grp T$ in $\grp G$ let $\grp T^+$ (resp.~ $\grp T^-$) be the maximal subtorus of $\grp T^\inv$ (resp.~$\{t\in T:\inv(t)=t^{-1}\}$). Then $\grp T=\grp T^+ \grp T^-$. In particular, an element of $X^*(T)$ is determined by its restrictions to $\grp T^+$ and $\grp T^-$.

\section{Preliminaries on the symmetric subgroup}\label{prelim}

Note that $\inv$ induces an involution on the set $X_*(G)$ that we further denote by $\inv$, its fixed points are precisely the elements of $X_*(H)$.

\begin{lemma}\label{lem: par of H}
The collection of parabolic $F$-subgroups of $\grp H^\circ$ is the set of groups of the form $\grp P\cap \grp H^\circ$ where $\grp P$ is a $\inv$-stable parabolic $F$-subgroup of $\grp G$. 
\end{lemma}
\begin{proof}
A parabolic $F$-subgroup of $\grp H^\circ$ is of the form $\grp{P_{H^\circ}}(\lambda)$, where $\lambda\in X_*(H^\circ)\subset X_*(G)$. It follows by definition that $\grp{P_{H^\circ}}(\lambda) =   \grp{P_G}(\lambda)\cap \grp H^\circ $. Note further that $\theta(\lambda)=\lambda$ and therefore $\inv(\grp{P_G}(\lambda))=\grp{P_G}(\inv(\lambda))=\grp{P_G}(\lambda)$, i.e., $\grp{P_G}(\lambda)$ is a $\inv$-stable parabolic $F$-subgroup of $\grp G$. 

Conversely, suppose that $\grp P$ is a $\inv$-stable parabolic $F$-subgroup of $\grp G$. By \cite[Lemma 2.4]{MR1215304} there exists a maximal $\inv$-stable torus $\grp A$ of $\grp G$ contained inside $\grp P$. Now, by \cite[Lemma 3.3]{MR1215304} there exists $\lambda\in X_*(A^+)$ such that $\grp P = \grp{P_G}(\lambda)$. Since $A^+\subset H^\circ$, the $F$-subgroup $\grp P \cap \grp H^\circ = \grp{P_{H^\circ}}(\lambda)$ of $\grp H^\circ$ is parabolic.

\end{proof}

Fix a minimal parabolic $F$-subgroup $\grp{ P_0^H}$ of $\grp H^\circ$. Let $\grp P_1$ be minimal amongst the $\theta$-stable parabolic $F$-subgroups $\grp P$ of $\grp G$ such that  $\grp P \cap \grp H^\circ = \grp{P_0^H}$. It follows from Lemma \ref{lem: par of H} that $\grp P_1$ is in fact a minimal $\theta$-stable parabolic $F$-subgroup of $\grp G$.

We may choose the minimal parabolic $F$-subgroup $\grp P_0$ of $\grp G$ to be contained in $\grp P_1$. By \cite[Lemma 2.4]{MR1215304} we may and do further choose $\grp A_0$ to be $\inv$-stable. Thus $\inv$ acts on $X_*(A_0)$, $X^*(A_0)$, $\aaa_0$ and $\aaa_0^*$. 

Note that if $\alpha\in \roots^G$ has root vector $v\in \Lie(G)$ then 
\[
\Ad(\inv(a))\inv(v)=\inv(\Ad(a)v)=\alpha(a)\inv(v),\ a\in A_0 ,
\]
i.e., $\inv(v)$ is a root vector for $\inv(\alpha)$ and therefore $\inv$ acts on $\roots^G$ and maps the root space of $\alpha$ to that of $\inv(\alpha)$.

If $\grp P=\grp M\ltimes \grp U$ is a semi-standard $\inv$-stable parabolic $F$-subgroup of $\grp G$ then $\grp U$ and $\grp M$ are $\inv$-stable by the uniqueness of the semi-standard Levi decomposition. Thus, $\grp{A_M}$ is also $\inv$-stable. 

By \cite[Lemma 3.5]{MR1215304} $\grp A_0^+$ is a maximal $F$-split torus of $\grp H$ and the standard Levi decomposition $\grp P_1=\grp M_1 \ltimes \grp U_1$ is such that $\grp M_1$ is the centralizer of $\grp A_0^+$ in $\grp G$. 

Since $\inv$ acts as an involution on $\aaa_0$ it decomposes it into a direct sum of the $\pm 1$-eigenspaces which we
denote by $(\aaa_0)_\theta^\pm$. Similarly 
\[
\aaa_0^*=(\aaa_0^*)_\inv^+\oplus (\aaa_0^*)_\inv^-.
\]
The inclusion $X_*(A_0^+)\subseteq X_*(A_0)$ induces the identification 
\[
X_*(A_0^+)\otimes_\Z \R \simeq (\aaa_0)_\inv^+. 
\]
It is straightforward that the pairing $\sprod{\cdot}{\cdot}_G$ is $\inv$ invariant and therefore $(\aaa_0^*)_\inv^\pm$ is the dual of $(\aaa_0)_\inv^\pm$.
Thus, $\sprod{\cdot}{\cdot}_G$ restricted to $(\aaa_0^*)_\inv^+\times (\aaa_0)_\inv^+$ is the natural pairing $\sprod{\cdot}{\cdot}_H$ defined with respect to $A_0^+$.

Let $\grp P=\grp M\ltimes \grp U$ be a standard, $\inv$-stable parabolic $F$-subgroup of $\grp G$. Then $\inv$ acts as an involution on $\aaa_M$ and we obtain a decomposition $\aaa_M=(\aaa_M)_\inv^+\oplus (\aaa_M)_\inv^-$ to the $\pm1$-eigenspaces. A similar decomposition holds for the dual space and $(\aaa_M^*)_\inv^{\pm}$ is the dual of $(\aaa_M)_\inv^\pm$. We have $(\aaa_M)_\inv^+=\aaa_{A_M^+}$ and similarly for the dual space. We denote by $\lambda_\inv^\pm$ the projection of $\lambda\in \aaa_M^*$ to $(\aaa_M^*)_\inv^\pm$.

By \cite[Lemma 3.3]{MR1215304} every $\inv$-stable, semi-standard parabolic $F$-subgroup of $\grp G$ is of the form $\grp{P_G}(\lambda)$ for some $\lambda\in X_*(A_0^+)$.
In particular, there exists $\lambda_1\in X_*(A_0^+)$ such that $\grp P_1=\grp{P_G}(\lambda_1)$.

Let $\roots^H=R(A_0^+,H)$ be the root system of $H$, $\roots^{H,\pos}=R(A_0^+, P_0^H)$ the subset of positive roots and $\Delta^H$ the basis of simple roots with respect to $P_0^H$ and $W^H$ the Weyl group of $\roots^H$.

\subsection{The descendent root system}
Let $\roots^{G/H}=R(A_0^+,G)$
be the set of roots of $A_0^+$ in $\Lie(G)$. Clearly $\roots^H\subseteq \roots^{G/H}$. It follows from \cite[Proposition 5.7]{MR1215304} that, unless empty, $\roots^{G/H}$ is a root system with Weyl group $W^{G/H}=N_G(A_0^+)/C_G(A_0^+)$. (Recall that $C_G(A_0^+)=M_1$.) In particular, $W^H\subseteq W^{G/H}$. 
Furthermore, if  $\roots^{G/H}$ is empty then $H/A_G^+$ is compact.
This case will be of little interest to us and we assume in what follows that $H/A_G^+$ is isotropic.
We call $\roots^{G/H}$ the \emph{descendent root system}.

Since the root space decomposition of $\Lie(G)$ with respect to $A_0$ automatically provides a decomposition of $\Lie(G)$ into $A_0^+$-eigenspaces we have
\begin{equation}\label{lem: all rts rest}
\roots^{G/H}=\{\alpha|_{\grp A_0^+}: \alpha\in \roots^G\}\setminus \{0\}.
\end{equation}

\begin{lemma}\label{lem: pos rest}
Let $\alpha\in \roots^G$ be such that $\alpha|_{\grp A_0^+}\in \roots^H$. Then $\alpha\in \roots^{G,\pos}$ if and only if $\alpha|_{\grp A_0^+}\in \roots^{H,\pos}$.
\end{lemma}
\begin{proof}
Recall that $\lambda_1\in X_*(A_0^+)$ is such that $P_1=P_G(\lambda_1)$ and  $P_0^H=P_{H^\circ}(\lambda_1)$. Thus, $\alpha|_{\grp A_0^+}\in \roots^{H,\pos}$ if and only if $\sprod{\alpha|_{\grp A_0^+}}{\lambda_1}_H>0$. Our embedding of $X^*(A_0^+)$ in $(\aaa_0^*)_\inv^+$ identifies $\alpha|_{\grp A_0^+}$ with $\frac12(\alpha+\inv(\alpha))$. Since $\inv(\lambda_1)=\lambda_1$ it follows that 
\[
\sprod{\alpha}{\lambda_1}_G=\sprod{\frac12(\alpha+\inv(\alpha))}{\lambda_1}_G=\sprod{\alpha|_{\grp A_0^+}}{\lambda_1}_H. 
\]
Since $U_1\subseteq U_0$ it follows immediately that if $\alpha|_{\grp A_0^+}\in\roots^{H,\pos}$ then $\alpha\in \roots^{G,\pos}$. Conversely, if $\alpha\in \roots^{G,\pos}$ then $\sprod{\alpha}{\lambda_1}_G\ge 0$. If $\sprod{\alpha}{\lambda_1}_G=0$ then $\alpha\in R(M_1,A_0)$. But since $\grp A_0^+$ is contained in the centre of $\grp M_1$ this contradicts the fact that $\alpha|_{\grp A_0^+}$ is non-trivial. It follows that $\sprod{\alpha}{\lambda_1}_G>0$ and therefore that $\alpha|_{\grp A_0^+}\in \roots^{H,\pos}$.

\end{proof}

Note that
\begin{equation}\label{eq: inv on pm}
\inv(x)|_{A_0^+}=x|_{A_0^+} \ \ \ \text{and} \ \ \ \inv(x)|_{A_0^-}=-x|_{A_0^-} \ \ \ \text{for all} \ \ \ x\in X^*(A_0). 
\end{equation}
It follows that 
\begin{equation}\label{eq: zero cond}
x+\inv(x)=0 \ \ \ \text{if and only if} \ \ \ x|_{A_0^+}=0.
\end{equation}

Let 
\[
\dnull=\{\alpha\in \Delta^G:\inv(\alpha)=-\alpha\}\mathop{=}\limits^{\eqref{eq: zero cond}}\{\alpha\in \Delta^G:\alpha|_{A_0^+}=0\}.
\] 
Let $X_0$ be the subgroup of $X^*(A_0)$ generated by $\dnull$. Also set
\[
\dnotnull=\Delta^G\setminus \dnull.
\]
\begin{lemma}\label{lem: inv acts}
For every $\alpha\in \dnotnull$ there exist $\beta\in \dnotnull$ and $x\in X_0$ such that $\inv(\alpha)=\beta+x$.
\end{lemma}
\begin{proof}
It follows from the definitions that $X_0$ is $\inv$-stable. Thus, the action of $\inv$ on $X^*(A_0)$ induces an action (that we still denote by $\inv$) as an involution on $\Gamma:=X^*(A_0)/X_0$.

Let $\alpha\in \dnotnull$. If $\inv(\alpha)=\alpha$ then $\beta=\alpha$, $x=0$ and we are done. Assume that $\inv(\alpha)\ne \alpha$. Let $v\in \Lie(G)$ be a root vector for $\alpha$. Then $\inv(v)$ is a root vector for $\inv(\alpha)$ and by our assumption $v$ and $\inv(v)$ are linearly independent. It follows from \eqref{eq: inv on pm} that $v+\inv(v)\in \Lie(G)^\inv=\Lie(H)$ is a root vector for the root $\alpha|_{A_0^+}\in \roots^H$. By Lemma \ref{lem: pos rest} $\alpha|_{A_0^+}\in \roots^{H,\pos}$ and $\inv(\alpha)\in \roots^{G,\pos}$.

Let $x\mapsto \bar x$ be the projection of $X^*(A_0)$ to $\Gamma$ and let $\dnotnull=\{\alpha_1,\dots,\alpha_t\}$. Clearly, $\{\bar\alpha_1,\dots,\bar\alpha_t\}$ are $\Z$-linearly independent in $\Gamma$. Since $\inv(\alpha_i)\in \roots^{G,\pos}$ for all $i$, it follows that there exists $M=(n_{i,j})\in M_t(\Z)$, a matrix of non-negative integers, such that 
\[
\overline{\inv(\alpha_i)}=\sum_{j=1}^t n_{i,j} \bar\alpha_j.
\]
Since $\inv$ is an involution we get that $M^2=I_t$ is the identity matrix. It is now straightforward that $M$ is a permutation matrix. The lemma follows.
\end{proof}

Let 
\[
\Delta^{G/H}=\{\alpha|_{A_0^+}:\alpha\in \dnotnull\}=\{\alpha|_{A_0^+}:\alpha\in \Delta^G\}\setminus\{0\}\subset X^*(A_0^+) .
\] 

\begin{proposition}\label{prop: HW basis}
The set $\Delta^{G/H}$ is a basis of simple roots for the descendent root system $\roots^{G/H}$. 
\end{proposition}
\begin{proof}
Let $\beta=\alpha|_{A_0^+}\in \roots^{G/H}$ with $\alpha \in \roots^G$ (see \eqref{lem: all rts rest}). Then either $\alpha$ or $-\alpha$ is a linear combination with positive integer coefficients of elements of $\Delta$. Restricting to $A_0^+$ we get that, respectively, $\beta$ or $-\beta$ is a linear combination with positive integer coefficients of elements of $\Delta^{G/H}$. To prove the proposition we therefore only need to show that $\Delta^{G/H}$ is linearly independent. Set $\Delta^{G/H}=\{\beta_1,\dots,\beta_t\}$ and fix $\alpha_1,\dots,\alpha_t\in \dnotnull$ so that $\beta_i=\alpha_i|_{A_0^+}$, $i=1,\dots,t$. Let $\alpha_i'\in \dnotnull$ be given by Lemma \ref{lem: inv acts} so that $\inv(\alpha_i)-\alpha_i'\in X_0$. After rearrangement we may assume that there exist $k$, $0\le k\le t$ such that $\alpha_i'=\alpha_i$ if and only if $i\le k$. Note that $\{\alpha_i:i=1,\dots,t\} \cup \{\alpha_i': k<i\le t\}$ is a subset of exactly $2t-k$ elements in $\dnotnull$.

Suppose that $x_1\beta_1+\cdots+x_t\beta_t=0$, $x_1,\dots,x_t\in \R$ and let $\gamma=x_1\alpha_1+\cdots+x_t\alpha_t$. Then $\gamma|_{A_0^+}=0$ and by \eqref{eq: zero cond} $\gamma+\inv(\gamma)=0$. Therefore 
\[
\sum_{i=1}^k 2x_i \alpha_i+\sum_{i=k+1}^t x_i(\alpha_i+\alpha_i')\in X_0.
\]
From the linear independence of $\Delta^G$ it follows that $x_i=0$ for all $i$. The proposition follows.
\end{proof}

Note that our identifications give an action of the Weyl group $W^{G/H}$ on the vector space $(\aaa_0)_\inv^+$ and on its dual $(\aaa_0^*)_\inv^+$.

\begin{corollary}\label{cor: cone}
We have 
\begin{enumerate}
\item\label{part: roots} $\Delta^H\subseteq \ccone(A_0^+,\Delta^{G/H})$; 

\item\label{part: cone inc} $[(\aaa_0)_\inv^+]^{\Delta^{G/H},\wpos}\subseteq [(\aaa_0)_\inv^+]^{\Delta^H,\wpos}$ and hence $X_*(A_0^+)^{\Delta^{G/H},\wpos}\subseteq X_*(A_0^+)^{\Delta^H,\wpos}$; 

\item\label{part: cone deco} The set 
\[
[W^{G/H} / W^H]:=\{w\in W^{G/H}:w^{-1}[(\aaa_0)_\inv^+]^{\Delta^{G/H},\pos}\subseteq [(\aaa_0)_\inv^+]^{\Delta^H,\pos}\} 
\]
forms a complete set of representatives for $W^{G/H} / W^H$ and 
\[
X_*(A_0^+)^{\Delta^H,\wpos}=\cup_{w\in [W^{G/H} / W^H]} w^{-1} X_*(A_0^+)^{\Delta^{G/H},\wpos};
\]

\item\label{part: cone with w}
For every $w\in [W^{G/H} / W^H]$, $w(\roots^{H,\pos})\subseteq  \ccone(A_0^+,\Delta^{G/H})$.

\end{enumerate}
\end{corollary}
\begin{proof}
Since $\roots^H\subseteq \roots^{G/H}$ it follows from \eqref{lem: all rts rest} and Lemma \ref{lem: pos rest} that every element of $\roots^{H,\pos}$ is a restriction to $A_0^+$ of an element of $\roots^{G,\pos}$. In particular, if $\beta=\alpha|_{A_0^+}\in \Delta^H$ with $\alpha\in \roots^{G,\pos}\subseteq \ccone(A_0,\Delta^G)$ then writing $\alpha$ as a positive linear combination of elements of $\Delta^G$ and restricting to $A_0^+$ shows that $\beta\in  \ccone(A_0^+,\Delta^{G/H})$. This shows part \eqref{part: roots}.

Part \eqref{part: cone inc} is straightforward from part \eqref{part: roots}. 

Recall that $\roots^H\subseteq \roots^{G/H}$ are root systems in $(\aaa_0^*)_\inv^+$. For $\lambda\in (\aaa_0^*)_\inv^+$ let 
\[
\h_\lambda=\{x\in (\aaa_0)_\inv^+:\sprod{\lambda}{x}=0\}. 
\] 
We have the Weyl chamber decomposition in the dual space
\[
(\aaa_0)_\inv^+\setminus (\cup_{\alpha\in \roots^H}\h_\alpha)=\bigsqcup_{w\in W^H} w [(\aaa_0)_\inv^+]^{\Delta^H,\pos}
\]
with respect to the root system $\roots^H$. The union is of connected components.
By Proposition \ref{prop: HW basis} we similarly have a decomposition
\[
(\aaa_0)_\inv^+\setminus (\cup_{\alpha\in \roots^{G/H}}\h_\alpha)=\bigsqcup_{w\in W^{G/H}} w [(\aaa_0)_\inv^+]^{\Delta^{G/H},\pos}
\]
with respect to the root system $\roots^{G/H}$.

 Since $\cup_{\alpha\in \roots^H}H_\alpha\subseteq \cup_{\alpha\in \roots^{G/H}}H_\alpha$, any connected component of $(\aaa_0)_\inv^+\setminus (\cup_{\alpha\in \roots^H}H_\alpha)$ is contained in a connected component of $(\aaa_0)_\inv^+\setminus (\cup_{\alpha\in \roots^{G/H}}H_\alpha)$. In particular, taking closures we have
\[
 [(\aaa_0)_\inv^+]^{\Delta^H,\wpos}=\cup_{w\in [W^{G/H} / W^H]} w^{-1}[(\aaa_0)_\inv^+]^{\Delta^{G/H},\wpos}
\]
and part \eqref{part: cone deco} follows.

Finally, for all $w\in [W^{G/H} / W^H]$, $\alpha\in \roots^{H,\pos}$ and $\lambda \in [(\aaa_0)_\inv^+]^{\Delta^{G/H},\wpos}$ we have
\[
\langle w(\alpha), \lambda \rangle = \langle \alpha , w^{-1}(\lambda) \rangle \geq 0. 
\]
Note, that $[(\aaa_0)_\inv^+]^{\Delta^{G/H},\wpos}$ and $\ccone(A_0^+,\Delta^{G/H})$ are both closed convex cones in Euclidean spaces, in the sense that they are closed under linear combinations with positive coefficients. Hence, by duality of convex cones, we have

\[ 
\ccone(A_0^+,\Delta^{G/H}) = \left\{ x \in (\aaa_0^*)_\inv^+\;:\; \sprod{\alpha}{\lambda} \geq 0\,,\;\forall \lambda\in [(\aaa_0)_\inv^+]^{\Delta^{G/H},\wpos}\right\}.
\]
The corollary follows.
\end{proof}

\begin{lemma}\label{lem: cone as union}
\begin{enumerate}
\item \label{part: rank}The dual lattices $X^*=X^*(A_0^+/A_G^+)$ and $X_*=X_*(A_0^+)/X_*(A_G^+)$ are of rank $\abs{\Delta^{G/H}}$. 
\item \label{part: dual}There exists a set $\{y_\alpha:\alpha\in \Delta^{G/H}\}\subset X_*(A_0^+)$ such that
\[
\sprod{\alpha}{y_\alpha}>0 \ \ \ \text{and}\ \ \ \sprod{\alpha}{y_\beta}=0 \ \ \ \text{for all}\ \ \ \alpha\ne\beta \text{ in }\Delta^{G/H}.
\]
\item \label{part: cone} For such a set $\{y_\alpha:\alpha\in \Delta^{G/H}\}$, let $Y$ be the subgroup of $X_*$ generated by the images of the $y_\alpha$'s and $Y^{\wpos}$ be the subset of $Y$ given by images of elements of the form $\sum_{\alpha\in \Delta^{G/H}} n_\alpha y_\alpha$ with $n_\alpha\in \Z_{\ge 0}$.

Then $Y$ is of finite index in $X_*$ and there exists a complete set of representatives $E$ for $X_*/Y$ so that we have the disjoint union
\[
X_*(A_0^+)^{\Delta^{G/H},\wpos}/X_*(A_G^+)=\bigsqcup_{e\in E} e+Y^{\wpos}.
\]
\end{enumerate}
\end{lemma}
\begin{proof}
By definition we have 
\[
\cap_{\beta\in \Delta^{G/H}} \ker \beta\subseteq\cap_{\alpha\in \Delta^G} \ker \alpha.
\]
Hence, since $A_G$ is the connected component of $\cap_{\alpha\in \Delta^G} \ker \alpha$, we also have that $A_G^+$ is the connected component of $\cap_{\beta\in \Delta^{G/H}} \ker \beta$. 

It follows that $\Delta^{G/H}$ embeds into $X^*$ and its image is a basis of the $\Q$-vector space $X^* \otimes_\Z \Q$. In particular part \eqref{part: rank} follows.
For each element of the dual basis (of $X_* \otimes_\Z\Q$) there is a positive integer that multiplies it into $X_*$. Choosing representatives mod $X_*(A_G^+)$ we obtain a set $\{y_\alpha:\alpha\in \Delta^{G/H}\}$ as in \eqref{part: dual}.
As its image in $X_*$ is a basis of $X_* \otimes_\Z\Q$ it follows that $Y$ is of finite index in $X_*$.

Let $E'$ be a complete set of representatives for $X_*/Y$ and let $c_\alpha=\sprod{\alpha}{y_\alpha}>0$, $\alpha\in \Delta^{G/H}$. For $e'\in E'$ let $m_{e',\alpha}\in \Z$ be minimal such that $\sprod{\alpha}{e'}+m_{e',\alpha}c_\alpha\ge 0$ and let $e=e'+\sum_{\alpha\in \Delta^{G/H}} m_{e',\alpha} y_\alpha$. Then $E=\{e: e'\in E'\}$ is still a complete set of representatives for $X_*/Y$. Note that
\[
\sprod{\alpha}{e}=\sprod{\alpha}{e'}+m_{e',\alpha}c_\alpha\ge 0
\]
hence $E\subseteq X_*(A_0^+)^{\Delta^{G/H},\wpos}$ and 
$\sprod{\alpha}{e}=\min_{x\in X_*(A_0^+)^{\Delta^{G/H},\wpos} \cap (e+Y)} \sprod{\alpha}{x}$ for all $\alpha\in \Delta^{G/H}$.
It follows that 
\[
X_*(A_0^+)^{\Delta^{G/H},\wpos}\cap ( e+ Y)=e + Y^{\wpos}
\]
and part \eqref{part: cone} follows.
\end{proof}

Let $\proj: \dnotnull \rightarrow \Delta^{G/H}$ be the surjective map defined by restriction to $A_0^+$.
\begin{lemma}\label{lem: inv st std par}
Let $I\subseteq \Delta^G$. Then $P_I$ is $\inv$-stable if and only if there exists $J\subseteq \Delta^{G/H}$ such that $I=\dnull\cup \proj^{-1}(J)$.
In particular, $P_1=P_{\dnull}$.
\end{lemma}
\begin{remark}\label{rmk: std st par}
Since $\proj$ is surjective, the map $J\mapsto \dnull \cup \proj^{-1}(J)$ from subsets of $\Delta^{G/H}$ to subsets of $\Delta^G$ is injective. It follows from the lemma that standard, $\inv$-stable parabolic $F$-subgroups of $\grp G$ are in order preserving bijection with subsets of $\Delta^{G/H}$.
\end{remark}
\begin{proof}
Assume that $P_I$ is $\inv$-stable. Recall that by \cite[Lemma 3.3]{MR1215304} we may take $\lambda_I\in X^*(A_0^+)\subseteq (\aaa_0)_\inv^+$ so that $P_I=P_G(\lambda_I)$.
By definition $\dnull\subseteq (\aaa_0^*)_\inv^-$ and therefore, $\sprod{\alpha}{\lambda_I}_G=0$ for all $\alpha\in \dnull$. As argued in the proof of Lemma \ref{lem: pos rest}, for $\alpha\in \dnotnull$ we have $\sprod{\alpha}{\lambda_I}_G=\sprod{\proj(\alpha)}{\lambda_I}_H$. It follows that
\[
I=\{\alpha\in \Delta^G: \sprod{\alpha}{\lambda_I}_G=0\}=\dnull\cup \proj^{-1}(J),
\]
where $J=\{\beta\in \Delta^{G/H}:\sprod{\beta}{\lambda_I}_H=0\}$.

Conversely, let $J\subseteq \Delta^{G/H}$ and $I=\dnull\cup \proj^{-1}(J)$. It follows from Proposition \ref{prop: HW basis} and Lemma \ref{lem: cone as union}\eqref{part: rank} that
there exists $\mu \in X_*(A_0^+)$ such that $\sprod{\beta}{\mu}_H=0$ if $\beta\in J$ and $\sprod{\beta}{\mu}_H>0$ if $\beta\in \Delta^{G/H}\setminus J$. Arguing as above we get that
$I=\{\alpha\in \Delta^G:\sprod{\alpha}{\mu}_G=0\}$. Therefore $P_I=P_G(\mu)$. As in Lemma \ref{lem: par of H} it follows that $P_I$ is $\inv$-stable. 
\end{proof}

For a standard, $\inv$-stable parabolic $F$-subgroup $\grp P=\grp M\ltimes \grp U$ of $\grp G$ let
\[
\Delta^{G/H}(M)=\{\beta|_{A_M^+}:\beta\in \Delta^{G/H}\}\setminus \{0\}=\{\alpha|_{A_M^+}:\alpha\in \Delta^G\}\setminus \{0\}.
\]
Let $J\subseteq \Delta^{G/H}$ and $I=\dnull\cup \proj^{-1}(J)$ be such that $\grp P=\grp P_I$.
\begin{lemma}\label{lem: rest rel rts to M}
Restriction to $A_M^+$ defines a bijection between $\Delta^{G/H}\setminus J$ and $\Delta^{G/H}(M)$.
Furthermore, $\Delta^{G/H}(M)$ is linearly independent.
\end{lemma}
\begin{proof}
Recall that 
\[
I=\Delta^M=\{\alpha\in \Delta^G:\alpha|_{A_M}=0\}. 
\]
Therefore  
\[
\Delta^{G/H}(M)=\{\beta|_{A_M^+}: \beta\in\Delta^{G/H}\setminus J\} \setminus \{0\}.
\]
Let $\Delta^{G/H}\setminus J=\{\beta_1,\dots,\beta_t\}$.
To conclude the lemma it is enough to show that for $x_1,\dots,x_t\in \R$ we have, if $x_1\beta_1+\dots+x_t\beta_t$ is trivial on $A_M^+$ then $x_i=0$ for all $i=1,\dots,t$.

If $\beta\in \Delta^{G/H}\setminus J$ then $\beta=\alpha|_{A_0^+}$ for some $\alpha\in \dnotnull\setminus I$.
Assume that $\sum_{i=1}^t x_i \beta_i|_{A_M^+}=0$.
Let $\alpha_i\in \dnotnull\setminus I$ be such that $\alpha_i|_{A_0^+}=\beta_i$ and let $\gamma=\sum_{i=1}^t a_i\alpha_i$.
Then $\gamma|_{A_M^+}=0$ and therefore by a standard argument that we already applied we have $(\gamma+\inv(\gamma))|_{A_M}=0$. Therefore, $\gamma+\inv(\gamma)$ is a linear combination of elements of $I=\Delta^M$.
On the other hand, let $\alpha_i'\in \dnotnull$ be given by Lemma \ref{lem: inv acts} so that $\inv(\alpha_i)-\alpha_i'\in X_0$. Since $\alpha_i$, $\inv(\alpha_i)$ and $\alpha_i'$ coincide on $A_0^+$, it follows that $\alpha_i'$ is not trivial on $A_M$ and therefore $\alpha_i'\in \Delta^G\setminus I$. Since $\dnull\subseteq I$, every element of $X_0$ is a linear combination of elements of $I$. It follows that $\sum_{i=1}^t x_i (\alpha_i+\alpha_i')$ is in the span of $I$. Arguing as in the proof of Proposition \ref{prop: HW basis}, by the linear independence of $\Delta^G$ it follows that $x_i=0$ for all $i$ and the lemma follows.  
\end{proof}

We call $\cone(A_M^+,\Delta^{G/H}(M))$ the cone of \emph{relatively positive} elements in $(\aaa_M^*)_\inv^+$.
Recall that 
\[
\aaa_0^*=(\aaa_0^M)^*\oplus (\aaa_M^*)_\inv^+\oplus  (\aaa_M^*)_\inv^-. 
\]
\begin{definition}\label{def: M pos}
An element $\lambda\in \aaa_0^*$ is called $M$-relatively positive (resp.~weakly positive) if its projection $(\lambda_M)_\inv^+$ to $(\aaa_M)_\inv^+$ is in $\cone(A_M^+,\Delta^{G/H}(M))$ (resp. $\ccone(A_M^+,\Delta^{G/H}(M))$).
\end{definition}
\begin{corollary}\label{cor: rel rest M}
With the above notation we have
\[
\Delta^{G/H}(M)=\{\alpha|_{A_M^+}:\alpha\in \Delta_M\}\setminus \{0\}.
\]
Thus, any $\lambda\in \cone(A_M,\Delta_M)$ is $M$-relatively positive and any $\lambda\in \ccone(A_M,\Delta_M)$ is $M$-relatively weakly positive.
\end{corollary}
\begin{proof}
It follows from Lemma \ref{lem: rest rel rts to M} that every element of $\Delta^{G/H}(M)$ is of the form $\beta|_{A_M^+}$ for some $\beta\in \Delta^{G/H}\setminus J$. Let $\alpha\in \Delta^G$ be such that $\alpha|_{A_0^+}=\beta$. Then $\alpha\not\in I$ and therefore $\alpha|_{A_M}\ne 0$, i.e., $\gamma:=\alpha|_{A_M}\in \Delta_M$ is such that $\gamma|_{A_M^+}=\beta|_{A_M^+}$. Conversely, if $\beta\in \Delta_M$ is such that $\beta|_{A_M^+}\ne 0$ then $\beta=\alpha|_{A_M}$ for some $\alpha\in \dnotnull$. Thus, $\gamma:=\alpha|_{A_0^+}\in \Delta^{G/H}$ is such that $\beta|_{A_M^+}=\gamma|_{A_M^+}$ and therefore $\beta|_{A_M^+}\in \Delta^{G/H}(M)$. The rest of the corollary is now straightforward.
\end{proof}
\section{$H$-integrability}\label{Hint}
In what follows we apply Lemma \ref{lem: vol cart} to $H^\circ$ with respect to the minimal parabolic subgroup $P_0^H$ and the maximal $F$-split torus $A_0^+$. Write $P_0^H=M_0^H \ltimes U_0^H$ where $M_0^H$ is the centralizer in $H^\circ$ of $A_0^+$ and therefore $M_0^H\subseteq M_1^\inv$.
Let
\[
C_0^{H,\wpos}=C_{A_0^+}^{\Delta^H,\wpos}.
\]

Choose a finite subset $F_0^H$ of $M_0^H$ in such a way that 
\[
H^\circ=\bigsqcup_{f\in F_0^H}\bigsqcup_{c\in C_0^{H,\wpos}}K^{H^\circ} fc K^{H^\circ}
\]
holds. We further insure that $F_0^H$ is such that Lemma \ref{lem: vol cart} holds for $H^\circ$ with $\mathcal{I}^H$ as a basis of open normal subgroups of $K^{H^\circ}$.

For a subset $X$ of $C_{A_0^+}$ let $[X]$ be its image under the projection to $C_{A_0^+}/C_{A_G^+}$.

Let $C^\infty(A_G^+\bs G)$ be the space of functions $\phi:G\rightarrow \C$ such that $\phi(ag)=\phi(g)$, $g\in G$, $a\in A_G^+$ and there exists an open subgroup $K_0$ of $G$ such that $\phi$ is bi-$K_0$-invariant.

\begin{proposition}\label{prop: H-int as sum}
Let $\phi\in C^\infty(A_G^+\bs G)$. Then the following conditions are equivalent:
\begin{enumerate}

\item \label{cond: H-int conv} $\int_{A_G^+\bs H} \abs{\phi(h)}\ dh<\infty$;

\item \label{cond: sum conv} $\sum_{s\in [C_0^{H,\wpos}]} \delta^{-1}_{P_0^H}(s)\abs{\phi(h_1sh_2)}<\infty$ for all $h_1,\,h_2\in H$.
\end{enumerate}
\end{proposition}
\begin{proof}
Since 
$C_{A_G^+}$ is cocompact in $A_G^+$,
condition \eqref{cond: H-int conv} holds if and only if $\int_{C_{A_G^+}\bs H} \abs{\phi(h)}\ dh<\infty$. Let $D$ be a (finite) set of representatives for $H/H^\circ$ and let
$K_0\in \mathcal{I}^H$ be such that $\phi(d \cdot)$ is bi-$K_0$-invariant for all $d\in D$. Let $E$ be a (finite) set of representatives for $K^{H^\circ}/K_0$. Then $K^{H^\circ} fc K^{H^\circ}=\cup_{e_1,\,e_2\in E} K_0 e_1fc e_2 K_0$ for all $f\in F_0^H$ and $c\in C_0^{H,\wpos}$. Hence
\[
H=\bigsqcup_{d\in D}\bigsqcup_{f\in F_0^H}\bigsqcup_{c\in C_0^{H,\wpos}}d K^{H^\circ} fc K^{H^\circ}
\]
and therefore
\begin{multline*}
\int_{C_{A_G^+}\bs H} \abs{\phi(h)}\ dh\le \sum_{d\in D}\sum_{f\in F_0^H} \sum_{e_1,\,e_2\in E} \sum_{s\in [C_0^{H,\wpos}]}\int_{K_0 e_1fs e_2 K_0} \abs{\phi(dh)}\ dh=\\
\sum_{d\in D}\sum_{f\in F_0^H} \sum_{e_1,\,e_2\in E} \sum_{s\in [C_0^{H,\wpos}]}\abs{\phi(de_1fs e_2)}\vol(K_0 e_1fs e_2 K_0).
\end{multline*}
Note further that
\[
\vol(K_0 e_1fs e_2 K_0)=\vol(e_1K_0 fs  K_0e_2)=\vol(K_0 fs  K_0)=\delta^{-1}_{P_0^H}(fs)\vol(K_0)
\]
where the identities follow respectively by the normality of $K_0$ in $K^{H^\circ}$, the invariance of the Haar measure on $H$ and Lemma \ref{lem: vol cart}.
Thus,
\[
\int_{C_{A_G^+}\bs H} \abs{\phi(h)}\ dh\le \vol(K_0)\sum_{d\in D}\sum_{f\in F_0^H} \delta^{-1}_{P_0^H}(f)\sum_{e_1,\,e_2\in E} \sum_{s\in [C_0^{H,\wpos}]}\delta^{-1}_{P_0^H}(s)\abs{\phi(de_1se_2)}.
\]
Since the sums over $d,\,f,\,e_1,\,e_2$ are finite clearly \eqref{cond: sum conv} implies \eqref{cond: H-int conv}.
Similarly, if 
\[
X=\cup_{s\in [C_0^{H,\wpos}]} K_0 s K_0
\]
then
\[
\vol(K_0)\sum_{s\in [C_0^{H,\wpos}]} \delta^{-1}_{P_0^H}(s)\abs{\phi(h_1sh_2)}=\int_{C_{A_G^+}\bs h_1 Xh_2} \abs{\phi(h)}\ dh\le \int_{C_{A_G^+}\bs H} \abs{\phi(h)}\ dh
\]
and therefore \eqref{cond: H-int conv} implies \eqref{cond: sum conv}.
\end{proof}

\subsection{Exponents}\label{exp}
Let $(\pi, V)$ be an admissible, smooth (complex valued) representation of $G$. For a parabolic subgroup $P=M\ltimes U$ of $G$, let $(r_P(\pi), r_P(V))$ denote the normalized Jacquet module of $\pi$ with respect to $P$ (see e.g. \cite{MR0579172}). It is an admissible representation of $M$. We say that a character $\chi$ of $A_M$ is an \emph{exponent} of $\pi$ along $P$, if it is an $A_M$-eigenvalue on $r_P(V)$, i.e., there exists $0\neq v\in r_P(V)$ such that $r_P(\pi)(a)v= \chi(a)v$, $a\in A_M$.
See \cite[VII.1.]{MR2567785} for a more detailed discussion of this definition.

If $\pi$ is of finite length then so is $r_P(\pi)$. In this case, the exponents are the restrictions to $A_M$ of the central characters of the irreducible components in a decomposition series for $r_P(\pi)$.

Let $\Exp_P(\pi)$ denote the set of all exponents of $\pi$ along $P$. 

\subsection{The Casselman pairing}
Let $\pi$ be an admissible representation of $G$ and let $\tilde\pi$ be its contragredient. For $v\in \pi$ and $\tilde v\in \tilde\pi$ the function
\[
\mc_{v,\tilde v}(g)=\tilde v(\pi(g) v) ,\ g\in G
\] 
is called a matrix coefficient of $\pi$. Let $\MC(\pi)$ be the space of all matrix coefficients of $\pi$. In his unpublished notes, Casselman developed a tool to study the asymptotics of matrix coefficients of $\pi$ in terms of matrix coefficients of Jacquet modules of $\pi$ \cite[\S 4]{CassNotes}.
We recall the results relevant to us. 

Let $\grp P=\grp M \ltimes \grp U$ be a standard parabolic subgroup of $G$ and let $\grp P^-$ be the opposite parabolic. Casselman defined an $M$-invariant pairing on $r_P(\pi) \times r_{P^-}(\pi)$ that identifies $r_{P^-}(\pi)$ as the contragredient of $r_P(\pi)$ (see e.g. \cite[VI.9.6.2]{MR2567785}). Let $v_P$ denote the projection of $v\in \pi$ to $r_P(\pi)$. It follows that for $v\in V$ and $\tilde v\in \Tilde \pi$ we have $\mc_{v_P,\tilde v_{P^-}}\in \MC(r_P(\pi))$. Moreover (see e.g. \cite[VI.9.6.5]{MR2567785}),
there exists $\epsilon>0$ such that 
\begin{equation}\label{eq: cass}
c_{v,\tilde v}(a)=\delta_P^{1/2}(a) \,c_{v_P,\tilde v_{P^-}}(a),\ a\in C_{A_M}^{\pos}(\epsilon).
\end{equation}

\subsection{A relative convergence criterion}
Let 
\[
\rho_0^G=\Re(\delta_{P_0}^{1/2})\in(\aaa_0^G)^* 
\]
and $\rho_M^G=(\rho_0^G)_M\in (\aaa_M^G)^*$ its projection with respect to a standard Levi subgroup $M$ of $G$. Note that if $P=M\ltimes U$ is a standard, $\inv$-stable parabolic subgroup of $G$ then $(\rho_M^G)_\inv^+=\Re(\delta_P^{1/2}|_{A_M^+})$.

\begin{proposition}\label{prop: conv sum H}
Let $\pi$ be an admissible representation of $G$ so that $A_G^+$ acts on $\pi$ as a unitary character and let $\omega$ be a character of $A_0^+/A_G^+$. The following are equivalent.
\begin{enumerate}
\item \label{cond: conv}For every $\mc\in \MC(\pi)$ we have
\[
\sum_{s\in [C_{A_0^+}^{\Delta^{G/H},\wpos}]} \abs{\mc(s)\omega(s)}<\infty;
\]

\item \label{cond: exp}For every standard, $\inv$-stable parabolic $F$-subgroup $\grp P=\grp M\ltimes \grp U$ of $\grp G$ and for every $\chi\in \Exp_P(\pi)$ we have $\Re(\chi)+\Re( \omega)+\rho_0^G$ is $M$-relatively positive. 
\end{enumerate}
\end{proposition}
\begin{proof}
Let $\{y_\alpha:\alpha\in \Delta^{G/H}\}$ be as in Lemma \ref{lem: cone as union}\eqref{part: dual}. In the notation of the lemma let
$t_\alpha=y_\alpha(\varpi)$, $\E=\{e(\varpi):e\in E\}$ and 
\[
S=\{y(\varpi): y\in Y^{\wpos}\}=\{\prod_{\alpha\in \Delta^{G/H}} t_\alpha^{n_\alpha}: n_\alpha\in \Z_{\ge 0}\text{ for all }\alpha\in \Delta^{G/H}\}.
\]
It follows from Lemma \ref{lem: cone as union}\eqref{part: cone}  that we have the disjoint union
\[
[C_{A_0^+}^{\Delta^{G/H},\wpos}]=\bigsqcup_{\epsilon\in \E} \epsilon S.
\]
For a subset $J\subset \Delta^{G/H}$ and a positive integer $N$ let
\[
S_J(N)_0 = \left\{\prod_{\alpha\in \Delta^{G/H}\setminus J} t_\alpha^{n_\alpha}:\; N<n_\alpha  \right\}\;,\quad S_J(N)_1 = \left\{\prod_{\alpha\in  J} t_\alpha^{n_\alpha}:\; 0\leq n_\alpha\leq N  \right\}
\]
and
\[
S_J(N) = S_J(N)_0S_J(N)_1\subseteq S.
\]
Note that $S_J(N)_1$ is a finite set.
Clearly, for any fixed $N$ we have the disjoint union 
\[
S=\bigsqcup_{J\subseteq \Delta^{G/H}} S_J(N)
\]
and therefore
\begin{multline*}
\sum_{s\in [C_{A_0^+}^{\Delta^{G/H},\wpos}]} \abs{\mc(s)\omega(s)}=\sum_{\epsilon\in \E}\sum_{J\subseteq \Delta^{G/H}} \sum_{s\in S_J(N)}  \abs{\mc(\epsilon s)\omega(\epsilon s)}=\\
\sum_{\epsilon\in \E}\sum_{J\subseteq \Delta^{G/H}} \sum_{t\in S_J(N)_1}\abs{\omega(\epsilon t)}\sum_{s\in S_J(N)_0}  \abs{\mc(\epsilon ts)\omega(s)}.
\end{multline*}
Since $\mc(\epsilon t\cdot)\in \MC(\pi)$ and the first three summations on the right hand side are over a finite set, we see that condition \eqref{cond: conv} is equivalent to the condition:

\begin{multline}\label{cond: equiv}
\text{ for every }\mc\in \MC(\pi)\text{ and }J\subseteq\Delta^{G/H}\text{ there exists }N>0 \text{ such that we have }\\ \sum_{s\in S_J(N)_0}  \abs{\mc(s)\omega(s)}<\infty.
\end{multline}
For $J\subseteq \Delta^{G/H}$ let $I=\dnull\cup \proj^{-1}(J)$ and $P=M\ltimes U=P_I$. 
Let $S_M$ be the lattice generated by $\{t_\alpha:\alpha\in \Delta^{G/H}\setminus J\}$. 
We further formulate the condition:
\begin{multline}\label{cond: conv geom}
\sum_{s\in S_J(N)_0} \delta_P^{1/2}(s)\abs{Q(s) \chi(s)\omega(s)}<\infty \text{ for all }N>0,\ J\subseteq \Delta^{G/H}, \\ 
\chi\in \Exp_P(\pi)\text{ and polynomials }Q \text{ on }S_M\text{ with complex coefficients}.
\end{multline}
Clearly \eqref{cond: conv geom} holds if and only if for all $J\subseteq \Delta^{G/H}$, $\chi\in \Exp_P(\pi)$ and $\alpha\in \Delta^{G/H}\setminus J$ we have $\delta_P^{1/2}(t_\alpha)\abs{\chi\omega(t_\alpha)}<1$. Note that $S_J(N)_0$ is contained in $A_M^+$ and that $\delta_{P_0}|_{A_M}=\delta_P|_{A_M}$.
By Lemma \ref{lem: rest rel rts to M} we get that \eqref{cond: exp} is equivalent to \eqref{cond: conv geom}.
It is therefore enough to show that conditions \eqref{cond: equiv} and \eqref{cond: conv geom} are equivalent.

Assume that condition \eqref{cond: conv geom} holds. Fix $\mc\in \MC(\pi)$ and $J\subseteq \Delta^{G/H}$ (so that $I=\dnull\cup \proj^{-1}(J)$ and $P=M\ltimes U=P_I$). Let $\tilde\mc\in \MC(r_P(\pi))$ be the matrix coefficient associated by the Casselman pairing and $\epsilon>0$ be given by \eqref{eq: cass} so that
\[
\mc(a)=\delta_P^{1/2}(a) \tilde\mc(a),\ a\in C_{A_M}^{\pos}(\epsilon).
\]
An element of $\Delta_M$ is of the form $\alpha|_{A_M}$ for some $\alpha\in \Delta^G\setminus I$. Hence $\alpha|_{A_0^+}\in \Delta^{G/H}\setminus J$. It therefore follows from the
definition of the sets $S_J(N)_0$ that there exists $N$ large enough so that $S_J(N)_0\subseteq C_{A_M}^{\pos}(\epsilon)$.
To show that condition \eqref{cond: equiv} holds it is therefore enough to show that
\[
\sum_{s\in S_J(N)_0}  \delta_P^{1/2}(s)\abs{\tilde\mc(s)\omega(s)}<\infty.
\]
A standard argument (see e.g. p. 332-3 in the proof of Casselman's criterion in \cite[Theorem VII.1.2]{MR2567785}) shows that there exist polynomials $Q_\chi$, $\chi\in \Exp_P(\pi)$ on $S_M$, only finitely many of which are non-zero, so that
\[
\tilde\mc(s)=\sum_{\chi\in \Exp_P(\pi)} Q_\chi(s) \chi(s),\ \ \ s\in S_M.
\]
Hence \eqref{cond: equiv} follows immediately from \eqref{cond: conv geom}.

Conversely, assume that \eqref{cond: conv geom} does not hold. Let $J\subseteq \Delta^{G/H}$, $\alpha\in \Delta^{G/H}\setminus J$ and, in the above notation, $\chi\in \Exp_P(\pi)$ be such that 
$\delta_P^{1/2}(t_\alpha)\abs{\chi\omega(t_\alpha)}\ge 1$. Then $\sum_{s\in S_J(N)_0}  \delta_P^{1/2}(s)\abs{\chi(s)\omega(s)}=\infty$ for all $N>0$.
Set $\mc=\mc_{v,\tilde v}$ where $v\in \pi$ is such that $v_P$ is an eigenvector of $A_M$ with eigenvalue $\chi$ (this realizes $\chi$ as an exponent of $\pi$ along $P$) and $\tilde v\in \tilde\pi$ is such that $\sprod{v_P}{\tilde v_{P^-}}=1$. Then, $\tilde\mc|_{A_M}=\chi$ and the above argument applying the Casselman pairing shows that for $N$ large enough
\[
\sum_{s\in S_J(N)_0} \abs{\mc(s)\omega(s)}=\sum_{s\in S_J(N)_0}  \delta_P^{1/2}(s)\abs{\tilde\mc(s)\omega(s)}=\infty.
\]
Thus, condition \eqref{cond: equiv} fails to hold. (Indeed, $S_J(N_1)_0\subseteq S_J(N_2)_0$ for $N_1<N_2$ and therefore, if condition \eqref{cond: equiv} holds then it is satisfied with $N$ arbitrarily large.) 
\end{proof}

\begin{definition}
We say that a smooth representation $\pi$ of $G/A_G^+$ is $H$-integrable if for any $\mc\in \MC(\pi)$ we have
\[
\int_{H/A_G^+} |\mc(h)| \ dh<\infty.
\]
\end{definition}

Let $\rho_0^H=\Re(\delta_{P_0^H}^{1/2})$ and recall that the set $[W^{G/H} / W^H]$ was defined in Corollary \ref{cor: cone}\eqref{part: cone deco}.
We can now formulate our main result.
\begin{theorem}\label{thm: main}
Let $\pi$ be an admissible representation of $G/A_G^+$. Then $\pi$ is $H$-integrable if and only if for any $\inv$-stable, standard parabolic subgroup $P=M\ltimes U$ of $G$ and any $\chi\in \Exp_P(\pi)$, the element $\Re(\chi)+\rho_0^G-2w(\rho_0^H)$ is $M$-relatively positive for all $w\in [W^{G/H} / W^H]$. 
\end{theorem}

\begin{proof}
Let $\mathcal{N}^{G/H}$ be a subset of $N_G(A_0^+)$ consisting of a choice of a representative $n$ for every element $w\in [W^{G/H} / W^H]$. 
Since every (left or right) translation by $G$ of an element of $\MC(\pi)$ is again in $\MC(\pi)$ it follows from Proposition \ref{prop: H-int as sum} (in its notation) that $\pi$ is $H$-integrable if and only if 
\begin{equation}\label{cond: H-sum conv}
\sum_{s\in [C_0^{H,\wpos}]} \delta^{-1}_{P_0^H}(s)\abs{\mc(s)}<\infty \text{ for all }\mc\in \MC(\pi).
\end{equation}
By Corollary \ref{cor: cone} we have
\[
[C_0^{H,\wpos}]=\mathop{\cup}\limits_{n\in \mathcal{N}^{G/H}} n^{-1}[C_{A_0^+}^{\Delta^{G/H},\wpos}]n
\]
and therefore,
\[
\sum_{s\in [C_0^{H,\wpos}]} \delta^{-1}_{P_0^H}(s)\abs{\mc(s)}<\infty 
\]
if and only if 
\[
\sum_{s\in [C_{A_0^+}^{\Delta^{G/H},\wpos}]} \delta^{-1}_{P_0^H}(n^{-1}sn)\abs{\mc(n^{-1}sn)}<\infty 
\]
for all $n\in \mathcal{N}^{G/H}$.
Note that $\mc(n^{-1}\cdot n)\in\MC(\pi)$ and that $\Re(\delta_{P_0^H}(n^{-1}\cdot n))=2w(\rho_0^H)$, when $n$ represents $w\in [W^{G/H} / W^H]$. It now follows from Proposition \ref{prop: conv sum H} (applied with $\omega=\delta^{-1}_{P_0^H}(n^{-1}\cdot n)|_{A_0^+}$) that \eqref{cond: H-sum conv} is equivalent to the condition in the statement of the theorem.

\end{proof}
\begin{remark}\label{remark: proj M1}
Recall from Definition \ref{def: M pos} that the condition that, $\lambda$ is $M$-relatively positive for any $\inv$-stable standard Levi subgroup $M$ of $G$, depends only on $\lambda_{M_1}$. It follows that $\rho_0^G$ may be replaced by $\rho_{M_1}^G$ in Theorem \ref{thm: main}. 

Furthermore, since $A_0^+$ is contained in $A_{M_1}$ (indeed, $M_1$ is the centralizer in $G$ of $A_0^+$) we have $(\aaa_0^*)_\inv^+\subseteq \aaa_{M_1}^*$ and therefore 
\[
(\aaa_{M_1}^*)_\theta^+=(\aaa_0^*)_\theta^+.
\]
\end{remark}

 \subsection{The relative test characters}

Theorem \ref{thm: main} points on the significance of the exponents 
\[
\rho^w_{G/H}:= (\rho_0^G )_\inv^+ - 2w(\rho_0^H) = (\rho_{M_1}^G )_\inv^+ - 2w(\rho_0^H)\in (\aaa_{M_1}^*)_\theta^+=(\aaa_0^*)_\theta^+
\]
for $w\in [W^{G/H} / W^H]$. We will now present means to compute these exponents using the action of $\inv$ on the various root data involved.

For $\alpha\in \roots^{G/H}$, let $L^G_\alpha$ (resp. $L^H_\alpha$) be the weight space of $\alpha$ in $\Lie(G)$ (resp. $\Lie(H)$). Thus $L_\alpha^H=0$ if $\alpha\notin \roots^H$. Set
\[
M^G_\alpha = \dim L^G_\alpha,\quad M^H_\alpha = \dim L^H_\alpha.
\]

Since $A_0^+$ is $\inv$-fixed, its adjoint action on $\Lie(G)$ commutes with the $\inv$-action. Thus, each $L^G_\alpha$ is a $\theta$-invariant subspace of $\Lie(G)$.

\begin{lemma}\label{lem: multip}
Let $\alpha\in \roots^{G/H}$ and set 
\[
m_{\inv,\alpha}= \Tr(\inv|_{L^G_\alpha}).
\]

\begin{enumerate}
\item\label{part: trace}
We  have $m_{\inv,\alpha}= 2M^H_\alpha-M^G_\alpha$.

\item\label{part: trace vanish}
If $\inv(\beta)\neq \beta$ for every $\beta \in \roots^G$ such that $\beta|_{A_0^+}=\alpha$, then $m_{\inv,\alpha}=0$.

\end{enumerate}
\end{lemma}

\begin{proof}
The linear involution $\inv$ on $L^G_\alpha$ decomposes the space into a sum of the eigenspaces related to the eigenvalues $1$ and $-1$. The $1$-eigenspace is precisely $L^G_\alpha \cap \Lie(G)^\theta = L^H_\alpha$. Thus, $m_{\inv,\alpha} = 1\cdot M^H_\alpha + (-1)\cdot (M^G_\alpha-M^H_\alpha)$.

Suppose that $\alpha$ is as in the assumption of \eqref{part: trace vanish}. Then there is an even number of elements of $\roots^G$ whose restriction to $A_0^+$ is $\alpha$ and we can enumerate them as $\{\beta_1,\ldots,\beta_k, \gamma_1,\ldots,\gamma_k\}$ with $\inv(\beta_i)=\gamma_i$. Thus, $L^G_\alpha$ admits a decomposition $L^G_\alpha = V_1\oplus V_2$ with $\inv(V_1)=V_2$ (indeed take $V_1$ to be the direct sum of root eigenspaces in $\Lie(G)$ with respect to $\{\beta_1,\ldots,\beta_k\}$ and similarly $V_2$ with respect to $\{\gamma_1,\ldots,\gamma_k\}$). Evidently, this implies that $\inv|_{L^G_\alpha}$ is of zero trace.

\end{proof}

Let $\roots^{G/H,\pos} := \roots^{G/H} \cap \overline\cone(A_0^+,\Delta^{G/H})$ be the set of positive roots in $\roots^{G/H}$.  These are the non-zero restrictions to $A_0^+$ of roots in $\roots^{G,\pos}$.

\begin{proposition}\label{prop: rel char as sum}
For every $w\in [W^{G/H} / W^H]$ we have
\[
\rho^w_{G/H} = -\frac12 \sum_{\alpha\in \roots^{G/H,\pos}} m_{\inv,w^{-1}(\alpha)}\, \alpha.
\]
\end{proposition}
\begin{proof}
Recall that $\delta_{P_0}(a)= |\det (\Ad(a)|_{\Lie(P_0)})|_F$, $a\in A_0$. Applied to $H$ this gives 
\[
\rho^H_0= \frac12 \sum_{\alpha\in \roots^{H,\pos}} M^H_\alpha\,\alpha.
\] 
Applied to $G$ and composed with the projection of $\rho^G_0$ to $(\aaa_0^*)_\inv^+$ we have 
\[
(\rho_0^G)_\inv^+= \frac12 \sum_{\alpha\in \roots^{G/H,\pos}} M^G_\alpha\, \alpha.
\]

Now, let $w\in [W^{G/H} / W^H]$ be given. Then,  
\begin{equation}\label{eq: rhoh}
w(\rho^H_0)= \frac12 \sum_{\alpha\in \roots^{H,\pos}} M^H_\alpha\, w(\alpha)=  \frac12 \sum_{\alpha\in w(\roots^{H,\pos})} M^H_{w^{-1}(\alpha)}\, \alpha = \frac12 \sum_{\alpha\in \roots^{G/H,\pos}} M^H_{w^{-1}(\alpha)}\, \alpha.
\end{equation}
The last equality is obtained as follows. By
Corollary \ref{cor: cone}\eqref{part: cone with w} we have $w(\roots^{H,\pos})\subseteq \roots^{G/H,\pos}$. The equality will therefore follow if we show that $M^H_\beta=0$ (i.e., that $\beta\notin \roots^H$) for $\beta\in w^{-1}(\roots^{G/H,\pos}) \setminus \roots^{H,\pos}$. 
Assume by contradiction that $-\beta\in \roots^{H,\pos}$. As above, by Corollary \ref{cor: cone}\eqref{part: cone with w} we have $-w(\beta)\in \roots^{G/H,\pos}$, i.e., both $\pm w(\beta)\in \roots^{G/H,\pos}$ which is a contradiction. 

Finally, there exists $n\in N_G(A_0^+)$ (a representative of $w^{-1}$) such that $Ad(n)(L^G_\alpha)= L^G_{w^{-1}(\alpha)}$ for all $\alpha\in \roots^{G/H}$. Hence, $M^G_{\alpha} = M^G_{w^{-1}(\alpha)}$ and we can write
\begin{equation}\label{eq: rhog}
(\rho_0^G)_\inv^+ = \frac12 \sum_{\alpha\in \roots^{G/H,\pos}} M^G_{w^{-1}(\alpha)}\, \alpha.
\end{equation}
The statement now follows from \eqref{eq: rhoh}, \eqref{eq: rhog} and Lemma \ref{lem: multip}.
\end{proof}

\section{Some special cases}\label{sect: special cases}

In this section we examine our criterion for $H$-integrability of matrix coefficients on certain symmetric spaces.
In \cite{1203.0039}, Sakellaridis and Venkatesh defined the notion of a strongly tempered spherical variety.
We recall the definition and make an analogous definition for square-integrable representations\footnote{In fact, the definition of Sakellaridis and Venkatesh is for $\grp{G/H}$, it is more convenient for us to consider a single $G$-orbit $G/H$.}.

\begin{definition}\label{def: strongly}
We say that $G/H$ is strongly tempered (resp.~strongly discrete) if every irreducible tempered (resp.~square-integrable) smooth representation $\pi$ of $G$ is $H$-integrable.
\end{definition}

We provide examples of families of symmetric spaces for which the above properties hold. In order to be able to apply Theorem \ref{thm: main} to this problem, we first need to recall Casselman's criterion for square integrability \cite[Theorem 4.4.6]{CassNotes} and a similar criterion for temperdness (see e.g. \cite[Proposition III.2.2]{MR1989693}. 
\begin{theorem}\label{thm: cass}
Let $\pi$ be an admissible representation of $G$ for which the centre of $G$ acts by a unitary character. Then $\pi$ is square-integrable (resp.~tempered) if and only if $\Re(\chi)\in \cone(A_M,\Delta_M)$ (resp.~$\Re(\chi)\in \ccone(A_M,\Delta_M)$), for any standard parabolic $F$-subgroup $P=M\ltimes U$ of $G$ and any $\chi\in \Exp_P(\pi)$. 
\end{theorem}

\begin{remark}
Set $\grp G=\grp H \times \grp H$ and $\inv(x,y)=(y,x)$, $x,\,y\in\grp H$ an involution on $\grp G$. Then $\grp H\simeq \grp G^\inv$ is embedded diagonally in $\grp G$.
Clearly, $\inv$-stable parabolic subgroups of $\grp G$ are in bijection with parabolic subgroups of $\grp H$, $\roots^{G/H}=\roots^H$ and $W^{G/H}=W^H$. Applying Theorem \ref{thm: main} to a representation of the form $\pi \otimes \tilde \pi$ of $G$, where $\pi$ is an admissible representation of $H$, recovers Casselman's criterion for square integrability of representations of $H$.
\end{remark}

It is straightforward from the definitions that an $M_1$-relatively (weakly) positive element of $(\aaa^*_0)_\inv^+$ is also $M$-relatively (weakly) positive for every standard $\inv$-stable Levi subgroup $M$. The following is therefore a straightforward consequence of Corollary \ref{cor: rel rest M} and Theorems \ref{thm: main} and \ref{thm: cass}.

\begin{corollary}\label{prop: rel chars}
If the relative test characters $\rho^w_{G/H}$ are $M_1$-relatively positive (resp.~weakly positive) for all $w\in [W^{G/H} / W^H]$, then $G/H$ is strongly tempered (resp.~strongly discrete).
\end{corollary}

\subsection{Galois symmetric spaces are strongly discrete}\label{sec: galois}

Let $E/F$ be a quadratic field extension. Let $\grp H$ be a connected, reductive $F$-group and $\grp G=\Res_{E/F}(\grp H_E)$ be the restriction of scalars from $E$ to $F$ of the group $\grp H$ considered as an $E$-group. Thus, $G\simeq \grp H(E)$. The Galois involution of $E/F$ defines an involution on $\grp G$ that we denote by $\inv$. We identify $\grp H$ with $\grp G^\inv$ and call $G/H$ a Galois symmetric space. 

Since $\grp H$ is defined over $F$, so are the Lie algebra $\Lie(\grp H)$ and the adjoint action on it. Hence, we have 
\[
\Lie(G)\simeq \Lie(\grp H)(E)=\Lie(H)\otimes_F E
\]
and the action of $h\in H$ is given as $\Ad(h)(v\otimes e)=\Ad(h)v \otimes e$, $v\in \Lie(H)$, $e\in  E$. It follows, that any eigenvalue of $\Ad(A_0^+)$ on $\Lie(G)$ is also an eigenvalue on $\Lie(H)$ and therefore $\roots^{G/H}=\roots^H$. In particular, $W^{G/H}=W^H$.

Since standard parabolic subgroups of $H$ are in bijection with subsets of $\Delta^H$, $\inv$-stable, standard parabolic subgroups of $G$ are in bijection with subsets of $\Delta^{G/H}$ (see Remark \ref{rmk: std st par}) and $\Delta^H=\Delta^{G/H}$ the map $\grp P\mapsto \grp P^\inv$ is a bijection between $\inv$-stable, standard parabolic $F$-subgroups of $\grp G$ and standard parabolic $F$-subgroups of $\grp H$ with inverse $\grp Q\mapsto \Res_{E/F}(\grp Q_E)$.
In particular, we have 
\[
P_1^\inv=P_0^H.
\]

The following follows from the proof of \cite[Lemma 2.5.1]{MR2010737} \footnote{the lemma is formulated in the global setting but the proof is the same in the $p$-adic case.}.
\begin{lemma}\label{lem: mod}
Let $\grp P$ be a $\inv$-stable, stanadrd parabolic $F$-subgroup of $\grp G$. Then $\delta_P^{1/2}|_{P^\inv}=\delta_{P^\inv}$. 
\end{lemma}
It follows that $(\rho_{M_1}^G)_\inv^+=2\rho_0^H$ and hence $\rho_{G/H}^{e}=0$ where $e$ is the identity in $W^{G/H}$. Hence, the following is immediate from Corollary \ref{prop: rel chars}.

\begin{corollary}\label{cor: str dis}
Every Galois symmetric space $G/H$ is strongly discrete.
\end{corollary}

We can also state the precise criterion inferred from an application of Theorem \ref{thm: main} to the Galois case.

\begin{theorem}\label{thm: Gal}
Let $G/H$ be a Galois symmetric space and let $\pi$ be an admissible representation of $G/A_G^+$. Then $\pi$ is $H$-integrable if and only if for any $\inv$-stable parabolic subgroup $P=M\ltimes U$ of $G$ and any $\chi\in \Exp_P(\pi)$, the element $\Re(\chi)$ is $M$-relatively positive. 

\end{theorem}

Assume now in addition that $A_0=A_0^+$. Then by \eqref{lem: all rts rest} $\roots^G=\roots^{G/H}=\roots^H$ and in particular $\Delta^G=\Delta^H$.
Thus, standard parabolic subgroups of $G$ are all $\inv$-stable and in particular $P_0=P_1$.
In paricular, for any standard parabolic subgroup $P=M\ltimes U$ of $G$ we have $A_M=A_M^+$ and $\Delta_M=\Delta^{G/H}(M)$.
The following is therefore immediate from Theorems \ref{thm: Gal} and \ref{thm: cass}.
\begin{corollary}
Assume that $G/H$ is a Galois symmetric space and $A_0=A_0^+$. Let $\pi$ be an admissible representation of $G/A_G$. Then $\pi$ is $H$-integrable if and only if $\pi$ is square-integrable.
\end{corollary}
 
\begin{remark}
The condition $A_0^+=A_0$ is automatically satisfied if $\grp H$ is $F$-split. Indeed,
in this case $\grp A_0^+$ is a maximal torus of $\grp H$. Therefore, the torus $A_0$ in $G\simeq \grp H(E)$ cannot be of higher rank.
\end{remark}

Many of the examples we consider are associated to a quadratic extension of $F$. Fix for the rest of this work a quadratic extension $E/F$ with Galois involution $\sigma$ and an element $\tau\in E$ such that $\sigma(\tau)=-\tau$.

\subsection{The symmetric space $\grp{GL}_n(F)/\grp{O}_J(F)$ is strongly tempered}\label{sec: orthog}
Let $\grp{G} = \grp{GL}_n$. 
Every symmetric matrix $J\in \GL_n$ defines an $F$-involution $\inv(g) = J\,{}^tg^{-1}J^{-1}$ on $\grp G$. Denote the associated orthogonal group by $\grp{O}_J=\grp G^\inv=\grp{H}$.

After $G$-conjugation if necessary, we may assume without loss of generality (see e.g.~\cite[\S 15.3.10]{MR2458469}) that $J$ is of the form
\[
\begin{pmatrix} & & w_r \\ & J_0 & \\ w_r & & \end{pmatrix}
\]
where $J_0\in \GL_{n-2r}$ defines an anisotropic quadratric form ($r$ is the Witt index of $J$) and $w_r\in \GL_r$ is the permutation matrix $(w_r)_{i,j}=\delta_{i,r+1-j}$.\footnote{ By the classification of quadratic forms over a p-adic field $n-2r\leq 4$ and the number of possible orthogonal groups in $G$, up to conjugation, is bounded by $2  \abs{F^\times/(F^\times)^2}$. See \cite[\S 63C]{MR0067163}. }
We may and do further assume that $J_0$ is diagonal. 

We choose the torus of diagonal matrices in $G$ to be the $\inv$-stable maximal $F$-split torus $A_0$. We Write $\epsilon_i\in \aaa_0^*$ for the character of $A_0$ that takes a diagonal matrix to its $i$-th entry and identify $\aaa_0^*\simeq\R^n$ by identifying $\{\epsilon_1,\dots,\epsilon_n\}$ with the standard basis of $\R^n$. Note that
\[
A_0^+=\{\diag (a_1,\ldots,a_r,1,\ldots,1,a_r^{-1},\ldots, a_1^{-1}): \quad a_i\in F^*,\ i=1,\dots,r\}.
\]
We write 
\[
\eta_i = \epsilon_i|_{A_0^+}\in (\aaa_0)_\inv^+.
\]

Let $\grp{P}_0$ be the Borel subgroup of upper triangular matrices in $\grp G$. For a decomposition $n_1+\ldots + n_k= n$ let $\grp{P}_{(n_1,\ldots,n_k)}=\grp{M}_{(n_1,\ldots,n_k)}\ltimes \grp{U}_{(n_1,\ldots,n_k)}$ be the associated standard parabolic subgroups of $\grp{G}$ with its standard Levi decomposition, where the Levi subgroup $\grp{M}_{(n_1,\ldots,n_k)}$ is isomorphic to $\grp{GL}_{n_1}\times\cdots\times\grp{GL}_{n_k}$.

Then $\grp{P}_1= \grp{P}_{(1,\ldots,1, 2n-r,1,\ldots,1)}= \grp{M}_1 \ltimes \grp{U}_1$ is a standard, minimal $\inv$-stable parabolic $F$-subgroup of $\grp{G}$. The intersection $\grp{P^H_0} = \grp{P}_1\cap \grp{H}^\circ$ is a minimal parabolic $F$-subgroup of $\grp{H}^\circ$. The root system 
\begin{equation}\label{eq: roots bcr}
\roots^{G/H}=\begin{cases} \{\pm(\eta_i\pm \eta_j):1\le i\ne j \le r\}\cup \{\pm\eta_i,\,\pm 2\eta_i:i=1,\dots,r\} & 2r<n \\ \{\pm(\eta_i\pm \eta_j):1\le i\ne j \le r\}\cup \{\pm 2\eta_i:i=1,\dots,r\} & 2r=n\end{cases}
\end{equation}
is of type $BC_r$ when $2r<n$ and of type $C_r$ when $2r=n$. We have
\begin{equation}\label{eq: simple bcr}
\Delta^{G/H} =\begin{cases} \{\eta_i-\eta_{i+1}\}_{i=1}^{r-1}\cup\{ \eta_r\}& 2r<n \\ \{\eta_i-\eta_{i+1}\}_{i=1}^{r-1}\cup\{ 2\eta_r\} & 2r=n.\end{cases}
\end{equation}
We write $E_{i,j}\subset \Lie(G)= \gl_n(F)$ for the one-dimensional subspace of matrices vanishing outside the $(i,j)$-th entry. These are the weight spaces for the roots in $\roots^G$. For integers $a\le b$ let $[a,b]=\{a,a+1,\dots,b\}$ be the corresponding interval of integers. Note that the action of $\inv$ on $\gl_n(F)$ (given by $\inv(X) = -J\,{}^tXJ^{-1}$) satisfies $\inv(E_{i,j})= E_{n+1-j,\,n+1-i}$ whenever $i,\,j\in [1,r]\cup [n+1-r,n]$ and $\inv(E_{i,j}) = E_{j,n+1-i}$ for $1\leq i\leq r$ and $r<j\leq n-r$. It easily follows that for $\alpha\in \roots^{G/H}\setminus \{2\eta_1,\dots,2\eta_r\}$ and every $\beta\in \roots^G$ such that $\beta|_{A_0^+}=\alpha$ we have $\inv(\beta)\ne \beta$.
Thus, by Lemma \ref{lem: multip}\eqref{part: trace vanish}, $m_{\inv,\alpha}=0$. Furthermore, $\inv$ acts by $-1$ on $L_{2\eta_i}^G=E_{i,\,n+1-i}$ and therefore $m_{\inv,2\eta_i}=-1$.

In case $n=2r$ ($H$ is an $F$-split orthogonal group), the root system $\roots^H$ is of type $D_r$, $\Delta^H= \{\eta_i-\eta_{i+1}\}_{i=1}^{r-1}\cup\{ \eta_{r-1}+\eta_r\}$ and $W^H$ is an index $2$ subgroup of $W^{G/H}$. It is easy to check that $[W^{G/H} / W^H] = \{e,\epsilon\}$, where $\epsilon$ is the simple reflection associated with the root $2\eta_r\in \Delta^{G/H}$ and $e$ is the identity. It is straightforward that $m_{\inv,\epsilon^{-1}(\alpha)} = m_{\inv,\alpha}$ for all $\alpha \in \roots^{G/H}$. It therefore follows from  Proposition \ref{prop: rel char as sum} that $\rho^{\epsilon}_{G/H}=\rho^{e}_{G/H}$.

Otherwise, when $2r<n$, $\roots^H$ is of type $B_r$, $\Delta^H= \{\eta_i-\eta_{i+1}\}_{i=1}^{r-1}\cup\{ \eta_r\}$ and $W^H = W^{G/H}$.

In all cases, combining this with Proposition \ref{prop: rel char as sum} the relative test characters are given by 
\begin{equation}\label{eq: is pos}
\rho^{w}_{G/H} = \sum_{i=1}^r \eta_i= \sum_{j=1}^{r-1} j\cdot (\eta_j-\eta_{j+1}) + r\cdot \eta_r, \ \ \ w\in [W^{G/H} / W^H].
\end{equation}
This is $M_1$-relatively positive by the second equality. Thus, from Corollary \ref{prop: rel chars} we deduce the following.

\begin{corollary}\label{cor: gl_n-o_n}
The symmetric space $\GL_n/ \O_J$ is strongly tempered for every symmetric matrix $J\in \GL_n$. 
\end{corollary}
\subsection{The symmetric space $\grp{U}_{J,E/F}(F)/\grp{O}_J(F)$ is strongly tempered}
We provide another family of strongly tempered symmetric spaces. The computation of relative test characters in the case at hand reduces to that of the previous subsection.
We therefore maintain all the notation defined in the previous subsection and use different letters to denote the symmetric space we now consider.

Recall that $E=F[\tau]/F$ is a quadratic extension with Galois involution $\sigma$.
We consider the following embedding of $\grp{O}_J$ as a the group of fixed points of an involution on the unitary group associated with $J$ and $E/F$.

Let $\grp{G'}= \Res_{E/F}(\grp{G}_E)$ (see Section \ref{sec: galois}) and consider $\sigma$ as the Galois involution on $\grp G'$. Note that the involution $\inv$ on $\grp G=(\grp G')^\sigma$ extends to an involution on $\grp G'$ by the same formula $\inv(g)=J \,{}^t g^{-1} J^{-1}$, $g\in \grp G'$ and that $\sigma$ and $\inv$ commute.
Let $\inv'=\inv\sigma=\sigma\inv$ and 
$\grp U=\grp{U}_{J,E/F}=(\grp G')^{\inv'}$ be the associated unitary group.

Note that $\sigma$ restricts to an involution on $\grp U$ and $\grp U^\sigma=\grp{O}_J=\grp H$. We consider now the symmetric space $U/H$.

From this explicit construction it is easy to see that there exists a $\sigma$-stable, maximal $F$-split torus $A_0^U$ of $U$ such that $A_0^+$ is the maximal $F$-split torus in $(A_0^U)^\sigma$. Furthermore, 
\[
\grp{P_1^U} = \Res_{E/F}(\grp{P}_{1,\ldots,1, 2n-r,1,\ldots,1}) \cap \grp{U} 
\]
is a minimal $\sigma$-stable parabolic $F$-subgroup of $\grp{U}$ such that $\grp{P_1^U} \cap \grp H^\circ=\grp{P_0^H}$. 

We consider $\Lie(U)$ as the $\inv'$-fixed subspace of $\Lie(G') \simeq \gl_n(E) = \gl_n(F) + \tau\cdot \gl_n(F)$. 
Thus,
\[
\Lie(U)=\{X+\tau Y:X,\,Y\in \gl_n(F),\,X=-J \,{}^tXJ^{-1},\,Y=J \,{}^tYJ^{-1}\}.
\]

By studying the adjoint action of $A_0^+$ on $\Lie(U)$ we observe that $\roots^{U/H} = \roots^{G/H}$ (where on both sides we view elements as characters on $A_0^+$) and $\Delta^{U/H} = \Delta^{G/H}$. Hence also $\roots^{U/H,\pos} = \roots^{G/H,\pos}$.
Furthermore, for every $\alpha\in \roots^{U/H,\pos} \setminus \{2\eta_1,\dots,2\eta_r\}$ there is a subspace $V_\alpha\subseteq \gl_n(F)$ (explicated bellow) so that 
$L_\alpha^U=L_\alpha^{U,+}\oplus L_\alpha^{U,-}$ where 
\[
L_\alpha^{U,+}=\{v+\inv'(v): v\in V_\alpha\} \text{ and }L_\alpha^{U,-}=\{v+\inv'(v): v\in \tau V_\alpha\}.
\]
For all such $\alpha$ we have $\dim L_\alpha^{U,+}=\dim L_\alpha^{U,-}$ and clearly $\sigma$ acts by $\pm1$ on $L_\alpha^{U,\pm}$ respectively. Therefore $m_{\sigma,\alpha}=0=m_{\inv,\alpha}$. Also $L_{2\eta_i}=\tau E_{i,n+1-i}$ is one dimensional and clearly $m_{\sigma,2\eta_i}=-1=m_{\inv,2\eta_i}$ for $i=1,\dots,r$.

It follows that $m_{\sigma,\alpha}=m_{\inv, \alpha}$ for all $\alpha \in \roots^{U/H}=\roots^{G/H}$. This allows us to argue verbatim as in Corollary \ref{cor: gl_n-o_n} to deduce the following.

\begin{corollary}
Let $E/F$ be a quadratic extension and $J\in \GL_n$ a symmetric matrix. Then the symmetric space $\U_{J,E/F}/\O_J$ is strongly tempered. 
\end{corollary}

For the sake of completeness, we provide here the above mentioned spaces $V_\alpha$ that complete the reduction of our computation to that of the previous subsection.
For $1\le i< j\le r$ we have
\[
V_{\eta_i-\eta_j}=E_{i,j}
\ \ \ \text{and}\ \ \  
V_{\eta_i+\eta_j}=E_{i,n+1-j}
\]
whereas if $2r<n$ for $i=1,\dots,r$ we have
\[
V_{\eta_i}=\mathop{\oplus}\limits_{j=r+1}^{n-r} E_{i,j}.
\]

\subsection{The symmetric space $\grp{GL}_{2n}(F)/\grp{GL}_n(E)$ is strongly discrete}\label{sec: more pairs}
Let $\grp{G} = \grp{GL}_{2n}$ and $\nu=\tau^2\in F$. Define the involution $\inv(g)=tgt^{-1}$ on $\grp G$ where
\[
t=\diag \left( \left(\begin{array}{cc} 0 & \nu^{-1} \\ 1 & 0 \end{array}\right),\ldots, \left(\begin{array}{cc} 0 & \nu^{-1} \\ 1 & 0 \end{array}\right) \right).
\]
Note that $H=G^\inv \simeq \grp{GL}_n(E)$. We can choose $A_0$ to be the diagonal torus in $G$. It is $\inv$-stable and
\[
A_0^+=\{
\left\{ \diag (a_1,a_1, a_2, a_2, \ldots, a_n,a_n)\;:\; a_i\in F^*,\,i=1,\dots,n\right\}.
\]
We can take $P_1=P_{(2,\ldots,2)}$ to be the minimal $\inv$-stable parabolic subgroup of $G$ so that $P_0^H=P_1\cap H$ is a minimal parabolic subgroup of $H=H^\circ$. 

We then have $\roots^{G/H} = \roots^H$ and $W^{G/H}=W^H$. For each $\alpha \in \roots^{G/H}$ there are four roots in $\roots^G$ such that $\beta|_{A_0^+} = \alpha$. The involution $\inv$ does not fix any of the four. Thus, by Lemma \ref{lem: multip}\eqref{part: trace vanish}, $m_{\inv, \alpha}=0$ for all $\alpha\in \roots^{G/H,\pos}$. In particular, the relative test character $\rho^{e}_{G/H}=0$. From Corollary \ref{prop: rel chars} we have the following.

\begin{corollary}
The symmetric space $\grp{GL}_{2n}(F)/ \grp{GL}_n(E)$ is strongly discrete.
\end{corollary}

\subsection{The symmetric space $\grp{Sp}_{2n}(F)/\grp{U}_{J,E/F}(F)$ is strongly tempered}
To describe an explicit realization of the symmetric space that we consider next it is convenient to maintain the notation of the previous subsection. 
For a symmetric matrix $J\in \GL_n$, we can embed the corresponding unitary group $\grp{U}_{J, E/F}$ in $\grp{Sp}_{2n}$ as follows. 
To $J= (a_{ij})$ we associate the anti-symmetric matrix $A_J\in \GL_{2n}$ whose 
whose $(i,j)$-th $2\times 2$ block is given by 
\[
\left(\begin{array}{cc} 0 & a_{ij} \\ -a_{ij} & 0 \end{array}\right).
\]
Let $\sigma(g)= A_J\,{}^tg^{-1}A_J^{-1}$ be the involution on $\grp G$ so that $\grp G^\sigma=\grp{Sp}_{A_J}\simeq\grp{Sp}_{2n}$. Note that the involutions $\sigma$ and $\inv$ commute, hence $\inv$ restricts to an involution on $\grp{Sp}_{A_J}$ and $\grp{Sp}_{A_J}^\inv\simeq \grp{U}_{J, E/F}$.
The group $\grp{U}_{w_n, E/F}$ is quasi-split over $F$. It is well known that if $n$ is odd then every unitary group is $\grp{GL}_n(E)$-conjugate to $\grp{U}_{w_n, E/F}$.
If $n$ is even then there are two conjugacy classes of non-isomorphic unitary groups determined by the norm class of the discriminant. We consider the two cases as follows.

Let $\grp G_1'=\grp{Sp}_{A_{w_n}}\simeq \grp{Sp}_{2n}$ and $\grp U_1=(\grp G_1')^\inv\simeq \grp U_{w_n,E/F}$. If $n$ is even let $\delta\in F^*$ be such that $\delta\det w_{n-2}\det w_n$ is not a norm from $E$ to $F$ and let 
\[
J_2=\begin{pmatrix} & & w_{n/2-1} \\ & d & \\ w_{n/2-1} & & \end{pmatrix}
\]
where $d=\diag(1,\delta)$. Set $\grp G_2'=\grp{Sp}_{A_{J_2}}\simeq \grp{Sp}_{2n}$ and $\grp U_2=(\grp G_2')^\inv\simeq \grp U_{J,E/F}$ the non-quasi-split unitary group.

In order to unify notation for the two cases at hand we set $J=w_n$ (resp.~$J=J_2$) and $\grp G'=\grp G'_1$ (resp.~$\grp G'=\grp G'_2$) and let $\grp U=(\grp G')^\inv$ be the corresponding unitary group.
We can choose the minimal $\inv$-stable parabolic subgroup $\grp P_1'$ of $\grp G'$ to be
\[
\grp P_1'=\begin{cases}  \grp P_{(2^{(n)})}\cap \grp G' & J=w_n \\  \grp P_{(2^{(n/2-1)},4,2^{(n/2-1)})}\cap \grp G' & J=J_2\end{cases}
\]
where 
$
(2^{(a)})=(\overbrace{2,\dots,2}\limits^a).
$
It contains a $\inv$-stable maximal $F$-split torus $\grp A'_0$ of $\grp G'$, such that $(\grp A'_0)^+$ is the maximal $F$-split torus of $\grp U$ such that 
\[
(A'_0)^+ = \{\diag ( a_1, a_1,\ldots, a_r,a_r,I_{2n-4r}, a_r^{-1}, a_r^{-1},\ldots,a_1^{-1},a_1^{-1})\;:\; a_i\in F^*,\,i=1,\dots,r\},
\]
where $r= \lfloor n/2 \rfloor$ in the quasi-split case, and $r=n/2-1$ in the non-quasi-split case. 
For our computation we recall that 
\[
\Lie(G')=\{X\in\gl_{2n}(F): -A_J \,{}^t X A_J^{-1}=X\}.
\]
The root system $\roots^{G'/U}$ is of the same type as in the example of subection \ref{sec: orthog}. Namely, $\roots^{G'/U}$ is of type $BC_r$ when $2r<n$
and of type $C_r$ when $2r=n$.
We may therefore denote the roots as in \eqref{eq: roots bcr} where 
$\eta_i$ is the character of $(A_0')^+$ that satisfies 
\[
\eta_i(\diag ( a_1, a_1,\ldots, a_r,a_r,I_{2n-4r}, a_r^{-1}, a_r^{-1},\ldots,a_1^{-1},a_1^{-1}))=a_i, \ \ \ i=1,\dots,r.
\] 
The simple roots $\Delta^{G'/U}$ are then given by \eqref{eq: simple bcr}.
Unlike in subsection \ref{sec: orthog}, we now have $\roots^U= \roots^{G'/U}$ and therefore $W^U= W^{G'/U}$ in all cases.

It is now a straightforward verification that for any $\alpha\in \roots^{G'/U}\setminus \{2\eta_1,\dots,2\eta_r\}$ there are four roots $\beta$ in $\roots^{G'}$ such that $\beta|_{(A'_0)^+} = \alpha$ and the involution $\inv$ fixes none of them. It follows from Lemma \ref{lem: multip}\eqref{part: trace vanish} that $m_{\inv, \alpha}=0$ for all $\alpha\in \roots^{G'/U}\setminus \{2\eta_1,\dots,2\eta_r\}$.

For $k=1,\dots,r$ the root space $L_{2\eta_k}^{G'}$ consists of matrices $X\in \Lie(G')$ such that the $(i,j)$-th $2\times 2$ block of $X$ is zero unless $i=k=n+1-j$ in which case it is of the form $\sm{a}{b}{c}{-a}$ for some $a,\,b,\,c\in F$. We denote such an element by $X_{a,b,c}$. Then 
\[
\inv(X_{a,b,c})=X_{-a,\nu^{-1}c,\nu b}
\]
and therefore $m_{\inv,2\eta_k}=-1$.
It now follows from Proposition \ref{prop: rel char as sum} that the relative test character is given by 
\[
\rho^e_{G'/U} = \sum_{i=1}^r \eta_i
\]
which is $M_1'$-relatively positive ($M_1'$ is the Levi subgroup of $P_1'$ containing $A'_0$) by the second equality in \eqref{eq: is pos}. Thus, from Corollary \ref{prop: rel chars} we deduce the following.

\begin{corollary}\label{cor: sp/u}
For every symmetric matrix $J\in \grp{GL}_n(F)$ the symmetric space $\Sp_{A_J}/ \U_{J, E/F}$ is strongly tempered. 
\end{corollary}

\begin{remark}
The split analogue of this case is the symmetric space $\Sp_{2n}/\GL_n$. It can be verified that it is strongly tempered for $n=1$ and strongly discrete for $n=2$.
However, the relative test characters are not all $M_1$-relatively weakly positive for $n\ge 3$. 
\end{remark}

\subsection{The symmetric spaces $\grp{GL}_{2n}(F)/\grp{GL}_n(F)\times \grp{GL}_n(F)$ and $\grp{GL}_{2n+1}(F)/\grp{GL}_n(F)\times \grp{GL}_{n+1}(F)$ are strongly discrete}
Let $\grp{G} = \grp{GL}_{n_1+n_2}$ and $\inv(g)=tgt^{-1}$, $g\in \grp G$ where $t=\diag( I_{n_1} , -I_{n_2})$. Then, $\grp{H}=\grp G^\inv \simeq \grp{GL}_{n_1}\times \grp{GL}_{n_2}$. 

Let $\grp{P}_1=\grp{P}_0$ be the standard Borel subgroup of upper triangular matrices and $\grp{A}_0$ the diagonal torus in $\grp G$. Then $\grp A_0=\grp A_0^+$ and therefore $\roots^G=\roots^{G/H}$ is of type $A_{n_1+n_2-1}$. For $1\leq i\ne j\leq n_1+n_2$ let $\alpha_{i,j}\in \roots^G$ be the root corresponding to the weight space $E_{i,j}$ defined as in \S \ref{sec: orthog}.
Then, $\Delta^{G/H} = \Delta^G = \{\beta_1,\ldots, \beta_{n_1+n_2-1}\}$. 
where $\beta_i= \alpha_{i,i+1}$, for $1\leq i\leq n_1+n_2-1$. 
We identify $W^G=W^{G/H}$ with the group $S_{n_1+n_2}$ of permutations on $\{1,\ldots,n_1+n_2\}$ so that $w(\alpha_{i,j}) = \alpha_{w(i),w(j)}$ for all $w\in W^G$. The set $[W^{G/H} / W^H]$ consists of all permutations that satisfy $w(i)<w(j)$ for all $1\leq i<j\leq n_1$ and $n_1+1\leq  i<j \leq n_1+n_2$.

\begin{lemma}\label{lem: lin period pos}
If either $n_2=n_1$ or $n_2=n_1+1$ then $\rho^w_{G/H}$ is $M_1$-relatively weakly positive for every $w\in [W^{G/H} / W^H]$. 
If $n_1=n_2=1$ then $\rho^w_{G/H}$ is $M_1$-relatively positive for every $w\in [W^{G/H} / W^H]$. 
\end{lemma}
\begin{proof}
For every $w\in [W^{G/H} / W^H]$, we write 
\[
\rho^w_{G/H}= a^w_1\beta_1+\ldots + a^w_{n_1+n_2-1} \beta_{n_1+n_2-1}
\] with half-integers $a^w_i$. Then $\rho^w_{G/H}$ is $M_1$-relatively weakly positive if and only if $a^w_k\geq 0 $ for all $1\leq k\leq n_1+n_2-1$. It is $M_1$-relatively positive when the inequalities ar?e strict.

Note that for $1\le i\ne j\le n_1+n_2$ we have
\[
m_{\inv, \alpha_{i,j}} = \begin{cases} 1 & i,j> n_1 \;\mbox{or} \; i,j\leq n_1 \\ -1 & \mbox{otherwise} \end{cases}
\]
 and that $\alpha_{i,j} = \beta_i + \beta_{i+1}+\cdots + \beta_{j-1}$ for all $i<j$.
 Set 
\[
d(w,k) = \#\{ (i,j)\;:\; 1\leq i\leq k <j\leq n_1+n_2,\; m_{\inv, \alpha_{w^{-1}(i),w^{-1}(j)}}=1\}.
\]
By Proposition \ref{prop: rel char as sum} we have 
\begin{multline*}
a_k^w=-\frac12\left[d(w,k)-\#\{ (i,j)\;:\; 1\leq i\leq k <j\leq n_1+n_2,\; m_{\inv, \alpha_{w^{-1}(i),w^{-1}(j)}}=-1\}\right]=
\\-\frac12\left[d(w,k)-(k(n_1+n_2-k)-d(w,k))\right]=
\frac{k(n_1+n_2-k)}2-d(w,k).
\end{multline*}
Note that translating by $w^{-1}$ we get that
\[
d(w,k)= \#\left\{ (i,j)\;:\;\begin{array}{l} 1\leq w(i)\leq k <w(j)\leq n_1+n_2\,, \\ \mbox{either }1\leq i < j \leq n_1 \mbox{ or } n_1+1 \leq i < j \leq n_1+n_2 \end{array} \right\}.
\]
Let 
\[
e_w = \begin{cases}\max\{1\leq i \leq n_1\;:\; w(i)\leq k \} & w(1)\le k \\ 0 & k<w(1).\end{cases}
\] 
Note that $e_w\le k$,
\[
k-e_w=\begin{cases}\max\{1\leq i \leq n_2\;:\; w(n_1+i)\leq k \} & w(n_1+1)\le k \\ 0 & k<w(n_1+1)\end{cases}
\]
and $\{w(i): 1\le i\le e_w\}\cup \{w(n_1+i): 1\le i\le k-e_w\}=\{1,\dots,k\}$.
It follows that
\[
d(w,k) = e_w(n_1-e_w) + (k-e_w)(n_2-(k-e_w)).
\]
Thus, in order to have $a_k^w\ge 0$ we need to show that
\begin{equation}\label{eq: convex}
\frac{k(k-(n_1+n_2))}2 \le e_w(e_w-n_1) + (k-e_w)((k-e_w)-n_2).
\end{equation}

Consider first the case $n_1=n_2$ and let $\phi(t)=t(t-n_1)$, $t\in \R$. It is a convex real function and therefore 
\[
2\phi(k/2)\le \phi(e_w)+\phi(k-e_w)
\]
(this is precisely the inequality \eqref{eq: convex}) and equality holds if and only if $e_w=k-e_w$.
This shows that $a_k^w\ge 0$ in this case. If in addition $n_1=1$ then $k=1$ and $e_w\ne k-e_w$. Thus in this case $a_1^w>0$ and $\rho_{G/H}^w$ is $M_1$-relatively positive.

Assume now that $n_2=n_1+1$. If $e_w=k-e_w$ then \eqref{eq: convex} is always an equality. Assume now that $e_w\ne k-e_w$
 and let $\psi(t)=t^2-t\left( \frac{t-e_w}{k-2e_w} + n_1\right)$, $t\in \R$. Again, it is a real function with non-negative second derivative and therefore 
\[
2\psi(k/2)\le \psi(e_w)+\psi(k-e_w)
\]
which is precisely the inequality \eqref{eq: convex}. The lemma follows.

\end{proof}
The following is now immediate from Lemma \ref{lem: lin period pos} and Corollary \ref{prop: rel chars}.
\begin{corollary}\label{cor: linear}
The symmetric spaces $\GL_{2n}/ \GL_n\times \GL_n$ and $\GL_{2n+1}/ \GL_n\times \GL_{n+1}$ are strongly discrete.
The symmetric space $\GL_2/ \GL_1\times \GL_1$ is strongly tempered.
\end{corollary}





\section{Non-vanishing}\label{sect: non-vanish}
For an $H$-integrable representation $\pi$ of $G$ and a vector $\tilde v$ in $\tilde\pi$ let $\ell_{\tilde v,H}$ be the linear form on $\pi$ defined by
\[
\ell_{\tilde v,H}(v)=\int_{H/A_G^+} \mc_{v,\tilde v}(h)\ dh.
\]
We write $\mathcal{L}^\pi_H=\{\ell_{\tilde v,H}:\tilde v\in \tilde\pi\}\subseteq \Hom_H(\pi,\C)$ for the subspace of $H$-invariant linear forms on $\pi$ emerging as integrals of matrix coefficients. 

Let $\grp X=\grp{G/H}$ be the $\grp{G}$-symmetric space associated with $\inv$. In \cite{1203.0039}, $\grp X$ is called strongly tempered if $G/H_z$ is strongly tempered (in the sense of Definition \ref{def: strongly}), for every $z\in X$ where $H_z$ is the stabilizer of $z$ in $G$. 
The statement of \cite[Theorem 6.4.1]{1203.0039} assumes that $\grp X$ is strongly tempered, but the proof considers a single $G$-orbit at a time. It therefore implies the following.
\begin{theorem}[Sakellaridis-Venkatesh]\label{thm: sv}
Assume that $\grp G$ is $F$-split and that $G/H$ is strongly tempered. 
If $\pi$ is an irreducible, square-integrable representation of $G$ then 
\[
\mathcal{L}^\pi_H=\Hom_H(\pi,\C).
\] 
If $\pi$ is a representation of $G$ parabolically induced from an irreducible, square-integrable representation of a Levi subgroup of $G$ then we have the implication
\[
\Hom_H(\pi,\C)\ne 0 \ \ \ \Rightarrow \ \ \ \mathcal{L}^\pi_H\ne 0.
\]
\end{theorem} 
The following is therefore an immediate consequence of Theorem \ref{thm: sv} and  Corollaries \ref{cor: gl_n-o_n}, \ref{cor: sp/u} and \ref{cor: linear}.
\begin{corollary}\label{thm: non-vanish}
For the following symmetric spaces $G/H$ and for every irreducible square-integrable representation $\pi$ of $G$ we have
\[
\mathcal{L}^\pi_H=\Hom_H(\pi,\C).
\]
\begin{enumerate}

\item $\GL_n/ \O_J$ for a symmetric matrix $J\in \GL_n$.\label{case: gl/o}
\item $\Sp_{2n}/ \U_{J, E/F}$ for a symmetric matrix $J\in \GL_n$.\label{case: sp/u}
\item $\GL_2/\GL_1\times\GL_1$.\label{case: gl2/torus}
\end{enumerate}
\end{corollary}
When $\grp G=\grp{GL}_n$, it follows from Zelevinsky's classification that representations of $G$ parabolically induced from irreducible square-integrable are precisely the irreducible tempered representations of $G$. 
We therefore also have the following.
\begin{corollary}
In cases \ref{case: gl/o} and \ref{case: gl2/torus} of Corollary \ref{thm: non-vanish}, for every irreducible tempered representation $\pi$ of $G$ we have
\[
\Hom_H(\pi,\C)\ne 0 \ \ \ \Rightarrow \ \ \ \mathcal{L}^\pi_H\ne 0.
\]
\end{corollary}
\begin{remark}
For $G/H=\GL_2/\GL_1\times \GL_1$, by multiplicity one, this implies that $\mathcal{L}^\pi_H=\Hom_H(\pi,\C)$ for every irreducible tempered representation of $\GL_2$.
\end{remark}

\bibliographystyle{amsalpha}
\bibliography{../Bibfiles/all}

\end{document}